\documentclass[12pt]{article}
\usepackage{amsmath,amsfonts,amssymb,amsthm}
\usepackage{color,enumerate}
\newtheorem{theorem}{Theorem}[section]
\newtheorem{proposition}[theorem]{Proposition}
\newtheorem{corollary}[theorem]{Corollary}
\newtheorem{lemma}[theorem]{Lemma}

\theoremstyle{remark}
\newtheorem{remark}[theorem]{Remark}

\newtheorem{example}[theorem]{Example}

\numberwithin{equation}{section}

\newcommand{\nc}{\newcommand}

\nc{\I}{{\mathbf 1}}
\nc{\bS}{{\mathbf S}}
\nc{\bN}{{\mathbf N}}
\nc{\bM}{{\mathbf M}}
\nc{\cB}{{\mathcal B}}
\nc{\cX}{{\mathcal X}}
\nc{\cM}{{\mathcal M}}
\nc{\R}{{\mathbb R}}
\nc{\N}{{\mathbb N}}
\nc{\Z}{{\mathbb Z}}
\nc{\BX}{{\mathbb X}}
\nc{\bx}{\mathbf{x}}
\nc{\bH}{\overline{H}}
\nc{\bP}{\overline{P}}

\DeclareMathOperator{\supp}{supp}
\DeclareMathOperator{\inter}{int}
\DeclareMathOperator{\Cov}{Cov}
\DeclareMathOperator{\Var}{Var}
\DeclareMathOperator{\conv}{conv}

\DeclareMathOperator{\cl}{cl}
\DeclareMathOperator{\rec}{rec}
\nc{\BP}{\mathbb{P}}
\nc{\BE}{\mathbb{E}}
\nc{\BQ}{\mathbb{Q}}
\nc{\eps}{\varepsilon}

\renewcommand{\emptyset}{\varnothing}
\newlength{\querylen}
\usepackage{fancybox}

\definecolor{olive}{rgb}{0.2, 0.6, .7}
\nc{\Green}[1]{{\color{olive} #1}}

\nc{\BT}{{\mathbb T}}
\nc{\cT}{{\mathcal T}}

\renewcommand{\subset}{\subseteq}

\begin{document}
\renewcommand{\thefootnote}{\fnsymbol{footnote}}
\author{G\"unter Last\footnotemark[1], Ilya Molchanov\footnotemark[2]}
\footnotetext[1]{guenter.last@kit.edu, 
Karlsruhe Institute of Technology, Institute for Stochastics, 76131 Karlsruhe, Germany. }
\footnotetext[2]{ilya.molchanov@unibe.ch, University of Bern,
Institute of Mathematical Statistics and Actuarial Science, CH-3012
 Bern, Switzerland}

\title{Efron type identities\\ for stopping sets and Poisson hulls}
\date{\today}
\maketitle

\begin{abstract}
  \noindent 
  We consider a Poisson process $\eta$ on a general Borel space with
  intensity measure $\lambda$ and a stopping set $Z$ depending on
  $\eta$.  Generalizing earlier results we prove a spatial Markov
  property for $Z$. In fact we establish this property for the larger class
  of finitely insertion tolerant random sets, which include pointwise
  limits of stopping sets. Using the spatial Markov property of
  $\eta$, we derive several distributional identities for the
  restrictions of $\eta$ and $\lambda$ to $Z$ and the complement of
  $Z$.  An important special case in Euclidean space is the convex
  hull of a finite Poisson process. In this case our results
  generalize classical (and also more recent) identities connecting
  the number of vertices and the volume of the convex hull.
\end{abstract}


\vspace{0.1cm}

\section{Introduction}

The convex hull $P$ of random points in $\R^d$ is a fundamental model
of stochastic geometry; see \cite[Section~8.2]{SW08}. Traditionally,
$P$ is either the convex hull of a finite number of independent and
identically distributed points, the binomial case, or the convex hull
of the points of a finite Poisson process, the Poisson case. Two
important geometric quantities are the number $N$ of vertices of $P$
and the value $\lambda(P)$, where $\lambda$ is either the probability
measure governing the distribution of the sample points in the
binomial case, or the intensity measure of the underlying Poisson
process in the Poisson case. Although the distributions of these
random variables can be expressed by multiple integrals involving
geometric functionals, such formulas are often difficult to use. In
the binomial case, Efron \cite{Efron65} discovered a simple identity
relating the expectations $\BE N$ and $\BE\lambda(P)$. This identity
was later extended to higher moments in \cite{Buchta05,Buchta23}.

Poisson versions of Efron's formula were derived in \cite{BeReitz15}
and in \cite[Eq.~(3.1)]{MR3366836}.
Because of the independence properties of Poisson processes, the
Poisson case is often more tractable. For instance, the Poisson
version of Efron's identity takes the form
\begin{align}
  \label{e:Efron}
  \BE N=\BE\lambda(\R^d\setminus P),
\end{align}
and follows easily from the Mecke equation. If $\lambda$ is not
diffuse, then $N$ has to count possible multiplicities, and the set
appearing on the right-hand side should be interpreted as the
complement of the convex hull with its vertices removed.

In this paper, we consider a Poisson process $\eta$ on a general Borel
space $(\BX,\cX)$ with intensity measure $\lambda$.  The observation
underlying our approach is that, for a finite Poisson process on
$\R^d$, the complement of its convex hull augmented by the vertices is
a stopping set. We use here the notion of a stopping set introduced in
\cite{LaPeYo23}, which includes the classical setting from
\cite{Zuyev99} as a special case.  The main aim of this paper is to
establish Efron type formulas for stopping sets and more general
Markov sets.  An important special case is a generator $\partial\eta$
and the associated hull $[\eta]$, as introduced and studied in
\cite{LM23}. Indeed Lemma~\ref{l:hullstopping} shows that the map
$\mu\mapsto[\mu]^c$, defined on the space of locally finite counting
measures, is a stopping set. 
A partial converse holds under a monotonicity assumption.

Section~\ref{sec:hull-generator} introduces stopping sets and the
broader class of finitely insertion tolerant sets $Z$.  Theorem
\ref{t:spatMarkov2} shows that the Markov property holds for the
latter sets, that is, the restriction of the Poisson process to the
complement of $Z$ is conditionally Poisson. This result generalizes
Theorems A.3 and A.6 from \cite{LaPeYo23}.  In particular it applies
to pointwise limits of stopping sets.  Example~\ref{ex:chull-problem}
shows that such limits do not need be stopping sets.  A notable
feature of our formulation is that the Markov property is expressed
after conditioning not only on the restriction $\eta_Z$ of the Poisson
process to $Z$, but also on the corresponding restriction $\lambda_Z$
of the intensity measure.  If $\lambda_Z$ is (almost surely)
measurable with respect to the corresponding restriction of $\eta$,
then the conditioning on the intensity measure can be omitted.  This
is true, for instance, if $Z$ is self-replicating and, in particular,
if $Z$ is a stopping set.  Theorem~\ref{thr:lambda-on-Z} shows that
this is also true in the convex hull case, provided the hull is closed
and does not contain any line.  Relations between generators, hulls,
and stopping sets are clarified in Appendix~A. However, an example in
Section~\ref{sec:counter} shows that the conditioning with respect to
$\lambda_Z$ cannot be dropped for general finitely insertion tolerant
sets.

Section~\ref{sec:identities} provides some properties of stopping sets
with finitely many Poisson points.  In Section~\ref{sec:markov-sets}
we further exploit the Markov property to derive several identities
which directly generalize the Efron identity for Poisson processes
from \cite{BeReitz15,MR3366836}.  Section~\ref{sec:ident-based-thinn}
uses independent thinning to express the generating function of the
number of revealed points through the intensity measure of a stopping
set. Further identities involving factorial moments and variances are
obtained in Section~\ref{s:furtherid}.

In many interesting examples, although not for ordinary convex hulls,
the hull operation satisfies the prime property. Informally, this means
that whether an added point affects the generator can be decided by
comparing it separately with each point of the underlying configuration.
This property leads to simplified Efron-type formulas, which are
derived in Section~\ref{sec:hulls-with-prime}.

Section~\ref{section:convexhulls} is devoted to convex hulls, including
the case where the convex hull of the points of a Poisson process is
unbounded. Finally, Section~\ref{section:ex} presents further examples
of stopping sets and the corresponding Efron-type identities.

\section{Stopping sets, hull and generator}
\label{sec:hull-generator}

Consider a Borel space $(\BX,\cX)$; see
\cite{kal17,LastPenrose17}.  We fix a \emph{localising ring}
$\cX_0\subset\cX$, see \cite{kal17}.  This is a ring with the
following two properties.  First, if $B\in\cX_0$ and $C\in\cX$, then
$B\cap C\in\cX_0$.  Second, there exists a sequence $B_n\in\cX_0$,
$n\in\N$, increasing to $\BX$ such that each set from $\cX_0$ is of
the form $C\cap B_n$ for some $C\in\cX$ and some $n\in\N$.
A measure $\nu$ is said to be \emph{locally finite}, if it is finite on
$\cX_0$.  The restriction of a measure $\nu$ to a set $B\in\cX$ is
defined by $\nu_B(\cdot):=\nu(\cdot\cap B)$. If $\nu$ is locally
finite, then so is $\nu_B$.  We denote by $\bN\equiv\bN(\BX)$ the
family of measures on $\BX$ which are integer-valued on $\cX_0$
and  equip $\bN$ with the smallest $\sigma$-algebra $\mathcal N$ making the 
maps $\mu\mapsto \mu(B)$ measurable for each $B\in\cX$.
We consider a Poisson process $\eta$ on  $(\BX,\cX)$ with
a locally finite  intensity measure $\lambda$, defined
on an underlying probability space $(\Omega,\mathcal{A},\BP)$,
see \cite{LastPenrose17}.

A map $Z\colon\bN\to\cX$ is said to be \emph{graph measurable} if
$(x,\mu)\mapsto \I\{x\in Z(\mu)\}$ is a measurable map on
$\BX\times\bN$ with the product $\sigma$-algebra. A graph measurable
map $Z$ is said to be a \emph{stopping set} if
\begin{equation}
  \label{eq:st-set-Z}
  Z(\mu)=Z(\mu_{Z(\mu)}+\psi_{Z(\mu)^c}),\quad \mu,\psi\in\bN, 
\end{equation}
see \cite[Appendix~A]{LaPeYo23}. In particular, $Z$ is
\emph{self-replicating}, that is, $Z(\mu)=Z(\mu_{Z(\mu)})$ for all
$\mu\in\bN$. The complement to $Z(\mu)$ is denoted by $[\mu]$ and is
said to be the \emph{hull operator}.

Many examples of stopping sets arise from a \emph{generator} which is
a map $\partial\colon\bN\to\bN$, as introduced in \cite{LM23}.  The
defining properties are listed in Appendix~A.  If $\partial$ is a
generator, then Lemma~\ref{l:hullstopping} shows that
$\mu\mapsto[\mu]^c$ is a stopping set and the associated hull operator
is given by
\begin{equation}
  \label{eq:hull-partial}
  [\mu]:=\{x\in\BX:\partial(\mu+\delta_x)=\partial\mu\},\quad \mu\in\bN.
\end{equation}
We refer to Appendix~A for general relations between stopping sets and
generators.
A key example (and a possible guiding star) in the case $\BX=\R^d$ is
based on the convex hull of the support of finite $\mu\in\bN$.
In this case the intensity measure $\lambda$ is finite so that one can
assume $\mu(\BX)<\infty$.  Then $\partial\mu$ is the restriction of
$\mu$ to the extreme points (vertices) of this convex hull and $[\mu]$
is the convex hull without the vertices. Convex
hulls of not necessarily finite measures are considered in
Section~\ref{section:convexhulls}.  

Given a graph-measurable $Z\colon\bN\to\cX$ we define (by a slight
abuse of notation) two point processes $\eta_Z:=\eta_{Z(\eta)}$ and
$\eta_{Z^c}:=\eta_{Z(\eta)^c}$.  Similarly we define measures
$\lambda_Z:=\lambda_{Z(\eta)}$ and
$\lambda_{Z^c}:=\lambda_{Z(\eta)^c}$.  Then $\lambda_Z$ and
$\lambda_{Z^c}$ are random measures, that is, random elements of the
space $\bM$ of all locally finite measures on $\BX$, equipped with the
standard $\sigma$-algebra.  Note that $\eta_Z+\eta_{Z^c}=\eta$, while
$\lambda_Z+\lambda_{Z^c}=\lambda$ is deterministic.  We let $\Pi_\nu$
denote the distribution of a Poisson process with intensity measure
$\nu$.  We say that a graph-measurable $Z\colon\bN\to\cX$ is a
\emph{Markov set} (with respect to $\eta$) if
\begin{equation}
  \label{eMarkov2}
  \BP(\eta_{Z(\eta)^c}\in \cdot\mid \eta_{Z(\eta)},\lambda_{Z(\eta)})
  =\Pi_{\lambda_{Z(\eta)^c}}(\cdot), \quad \BP\text{-a.s.}
\end{equation}
If $\lambda_Z$ is $\BP$-a.s.\ $\sigma(\eta_Z)$ measurable then
\eqref{eMarkov2} simplifies to
\begin{equation}
  \label{SMP}
  \BP\big(\eta_{Z(\eta)^c} \in\cdot\mid \eta_{Z(\eta)}\big)
  =\Pi_{\lambda_{Z(\eta)^c}}(\cdot),\quad \BP\text{-a.s.}
\end{equation}
This measurability holds if $Z$ is self-replicating and in particular
for stopping sets.  Theorem~\ref{t:spatMarkov2} below will show that a
stopping set is Markov.  In fact we will show the Markov property for
a larger class of graph-measurable sets, to be defined next.

A graph-measurable $Z\colon\bN\to\cX$ is said to be \emph{insertion
  tolerant} if, for $\mu,\psi\in\bN$,
\begin{align}
  \label{insertion-equivalence}
  \psi(Z(\mu+\psi))=0
  \quad \text{if and only if} \quad
  \psi(Z(\mu))=0,
\end{align}
and, whenever these equivalent conditions hold,
\begin{align}
  \label{insertion-equality}
  Z(\mu+\psi)=Z(\mu).
\end{align}
This is equivalent to the fact that
for $\mu,\psi\in\bN$, either condition
$\psi(Z(\mu+\psi))=0$ or $\psi(Z(\mu))=0$ implies $Z(\mu+\psi)=Z(\mu)$.
We call $Z$ \emph{finitely insertion tolerant} if
\eqref{insertion-equivalence} and \eqref{insertion-equality} hold for
all $\mu\in\bN$ and all finite $\psi\in\bN$.  We show in Appendix~B
that finite insertion tolerance implies self-replication at
\emph{finite} elements of $\bN$, while infinite insertion tolerance is
equivalent to the stopping set property; see
Lemma~\ref{lem:insertion-tolerance}.  The following simple example
shows that finite insertion tolerance does not imply
self-replication. In particular, a finitely insertion tolerant $Z$ does
not need to be a stopping set.

\begin{example}
  \label{ex:insertion-not-selfreplicating}\rm
  Let $\BX=\N$ and $B_n:=\{1,\ldots,n\}$, $n\in\N$. Then a subset of
  $\N$ is bounded if and only if it is finite.  Define
  $Z\colon\bN\to 2^{\N}$ by letting $Z(\mu):=\N$ if $\mu(\N)<\infty$
  and $Z(\mu):=\emptyset$ otherwise.  Clearly, $Z$ is
  graph-measurable. If $\psi\in\bN$ is finite, then
  $(\mu+\psi)(\N)<\infty$ if and only if $\mu(\N)<\infty$. Therefore,
  $Z(\mu+\psi)=Z(\mu)$.  In particular, $Z$ is finitely insertion
  tolerant.

  If $\mu\in\bN$ is infinite, then $Z(\mu)=\emptyset$, hence
  $\mu_{Z(\mu)}=0$.  Then
  \[
    Z(\mu_{Z(\mu)})=Z(0)=\N\ne\emptyset=Z(\mu),
  \]
  violating the self-replication property.
\end{example}

The following theorem is the main tool of this paper.  It generalizes
Theorems~A.3 and~A.6 in \cite{LaPeYo23}, see also Theorem~3.2 from
\cite{LM23}.

\begin{theorem}\label{t:spatMarkov2}
  Suppose that $Z\colon\bN\to\cX$ is finitely insertion tolerant. Then
  $Z$ is a Markov set.
\end{theorem}
\begin{proof}
  For $\mu\in\bN$, define $H(\mu):=(\mu_{Z(\mu)},\lambda_{Z(\mu)})$.
  Let $B\in\cX_0$ and let $f\colon\BX\to\R_+$ be measurable and vanish
  outside $B$.  Abbreviate $h:=1-e^{-f}$ and $\xi:=\eta_{Z(\eta)^c}$.

  Given $x_1,\dots,x_n\in\BX$, we write $\bx:=(x_1,\dots,x_n)$,
  $\delta_{\bx}:=\delta_{x_1}+\cdots+\delta_{x_n}$, and
  $[\bx]:=Z(\delta_{\bx})^c$.  Let $F\colon \bN\times\bM\to\R_+$ be
  bounded and measurable and $n\ge1$. The multivariate Mecke equation
  gives
  \begin{align*}
    \BE
    &\bigg[F(H(\eta))
      \int_{B^n}\prod_{i=1}^n h(x_i)\,\xi^{(n)}(d(x_1,\ldots,x_n))
      \bigg]\\
    &= \int_{B^n}\prod_{i=1}^n h(x_i)\,
      \BE\bigg[ F(H(\eta+\delta_{\bx}))
      \prod_{i=1}^n \I\{x_i\notin Z(\eta+\delta_{\bx})\}
      \bigg]\lambda^n(d\bx),
  \end{align*}
  Since $Z$ is finitely insertion tolerant,
  $\delta_{\bx}(Z(\eta+\delta_{\bx}))=0$ if and only if
  $\delta_{\bx}(Z(\eta))=0$, in which case
  $Z(\eta+\delta_{\bx})=Z(\eta)$.  If $\delta_{\bx}(Z(\eta))=0$, it
  follows that
  $ (\eta+\delta_{\bx})_{Z(\eta+\delta_{\bx})}=\eta_{Z(\eta)}$ and
  $\lambda_{Z(\eta+\delta_{\bx})}=\lambda_{Z(\eta)}$ and hence
  $H(\eta+\delta_{\bx})=H(\eta)$.  Therefore,
  \begin{equation}
    \label{markov-factorial}
    \BE \bigg[F(H(\eta))
    \int_{B^n}\prod_{i=1}^n h(x_i)\,\xi^{(n)}(d\bx) \bigg]
    = \BE\left[F(H(\eta))\big(\lambda_{Z(\eta)^c}(h)\big)^n\right].
  \end{equation}
  The same identity holds for $n=0$.

  Since $\xi(B)\le\eta(B)<\infty$ almost surely, Exercise~4.5 from
  \cite{LastPenrose17} (applied with $u:=1-e^{-f}$) yields that
  \begin{align}\label{e:036}
    e^{-\xi(f)}=\sum_{n=0}^\infty\frac{(-1)^n}{n!}
    \int_{B^n}\prod_{i=1}^n h(x_i)\,\xi^{(n)}(d\bx).
  \end{align}
  Multiplying both sides of \eqref{markov-factorial} with $(-1)^n/n!$
  and summing we obtain
  \begin{align}\label{e:039}
    \BE\big[F(H(\eta))e^{-\xi(f)}\big]
    = \BE\Big[F(H(\eta)) e^{-\lambda_{Z(\eta)^c}(h)}\Big].
  \end{align}
  On the left-hand side the interchange of expectation and summation
  is justified by
  \begin{align*}
    \sum_{n=0}^\infty\frac{1}{n!}\BE[\xi^{(n)}(B^n)]
    \le \sum_{n=0}^\infty\frac{1}{n!}\BE[\eta^{(n)}(B^n)]
    =\sum_{n=0}^\infty\frac{1}{n!}\lambda(B)^n<\infty
  \end{align*}
  and similarly for the right-hand side. 
  By bounded convergence this holds for any measurable
  $f\colon\BX\to\R_+$.  Using the explicit form of the Laplace
  functional of a Poisson process (and the definition of the function
  $h$), we can write \eqref{e:039} as
  \begin{align}\label{e:035}
    \BE\Big[F(H(\eta))e^{-\eta_{Z(\eta)^c}(f)}\Big]
    =\BE\bigg[F(H(\eta))\int e^{-\mu(f)}
    \,\Pi_{\lambda_{Z^c}}(d\mu)\bigg].
  \end{align}
  Assume $c:=\BE[F(H(\eta))]>0$ and define two probability measures
  $\BQ_1$ and $\BQ_2$ on $\bN$ by
  \begin{align*}
    \BQ_1:=c^{-1}\BE\Big[F(H(\eta))
    \I\{\eta_{Z(\eta)^c}\in\cdot\}\Big],\quad
    \BQ_2:=c^{-1}\BE\big[F(H(\eta))\Pi_{\lambda_{Z^c}}\big].
  \end{align*}
  By \eqref{e:035}, $\BQ_1$ and $\BQ_2$ have the same Laplace functional
  so that \cite[Proposition 2.10]{LastPenrose17} implies $\BQ_1=\BQ_2$.
  It then follows for any measurable $A\subset\bN$, that
  \begin{align*}
    \BE\big[F(H(\eta))\I\{\eta_{Z(\eta)^c}\in A\}\big]
    =\BE\big[F(H(\eta))\Pi_{\lambda_{Z^c}}(A)\big],
  \end{align*}
  which remains true if $\BE[F(H(\eta))]=0$.  This concludes the proof.
\end{proof}

Next, we discuss several special cases of Theorem~\ref{t:spatMarkov2}. 

\begin{corollary}\label{c:stoppingMarkov}
  Suppose that $Z$ is a stopping set. Then $Z$ is a Markov set.
\end{corollary}
\begin{proof}
  By Lemma~\ref{lem:insertion-tolerance}, $Z$ is finitely insertion
  tolerant and the result follows from Theorem~\ref{t:spatMarkov2}.
\end{proof}

\begin{corollary}\label{c:spatMarkov2}
  Let $(Z_n)$ be a  sequence of stopping sets. 
  Assume that $Z\colon \bN\to\cX$ satisfies
  \begin{align}\label{approxstopping}
    \lim_{n\to\infty}\I\{x\in Z_n(\mu)\}
    =\I\{x\in Z(\mu)\}, \quad (x,\mu)\in\BX\times\bN.
  \end{align}
  Then $Z$ is a Markov set.
\end{corollary}
\begin{proof}
  By Lemma~\ref{lem:insertion-tolerance}, $Z_n$, $n\geq1$, are
  finitely insertion tolerant, so that Lemma~\ref{lem:limit-insertion}
  shows that $Z$ is also finitely insertion tolerant. The result
  follows from Theorem~\ref{t:spatMarkov2}.
\end{proof}

If the convergence \eqref{approxstopping} is monotone increasing, then
$Z$ is a stopping set; see \cite[Theorem~A.6]{LaPeYo23}. Hence $Z$ is
self-replicating and \eqref{eMarkov2} simplifies to \eqref{SMP}. This
special case of Corollary~\ref{c:spatMarkov2} with
  $\BP(\eta(Z_n(\eta))<\infty)=1$ for all $n$ was proved in
\cite{LaPeYo23}.

If the convergence \eqref{approxstopping} is not monotone increasing,
then \eqref{approxstopping} does not imply that $Z$ is a stopping set;
see Example~\ref{ex:chull-problem}.  It might already fail to be
self-replicating.

We say that the convergence \eqref{approxstopping} is \emph{measurably
  closed}, if whenever $X_n$, $n\in\N$, are
$\sigma(\eta_{Z_n})$-measurable random variables converging pointwise
to some random variable $X$, then $X$ is almost surely
$\sigma(\eta_Z)$-measurable.

\begin{proposition}\label{t:spatMarkov3}
  Let the assumptions of Corollary~\ref{c:spatMarkov2} be satisfied
  and assume in addition that the convergence \eqref{approxstopping}
  is measurably closed.  Then the spatial Markov property holds in the
  form \eqref{SMP}.
\end{proposition}
\begin{proof}
We need to show that the conditioning with respect to $\lambda_Z$ in
\eqref{eMarkov2} can be dropped. Let $B\in\cX_0$. Then
$\lambda_Z(B)=\lim_{n\to\infty}\lambda(Z_n(\eta)\cap B)$.  Since the
$Z_n$ are self-replicating this is a limit of
$\sigma(\eta_{Z_n})$-measurable random variables. Therefore, we
obtain from the assumption that $\lambda_Z(B)$ is a.s.\
$\sigma(\eta_Z)$-measurable. A monotone class argument shows that
this remains true for the whole random measure $\lambda_Z$.
\end{proof}

\begin{corollary}
  Let $(Z_n)$ be a monotone decreasing sequence of stopping sets, and
  assume that $\BP(\eta(Z_1(\eta))<\infty)=1$. Set
  $Z(\mu):=\bigcap_{n\geq1}Z_n(\mu)$ for $\mu\in\bN$. Then the spatial
  Markov property holds in the form \eqref{SMP}.
\end{corollary}
\begin{proof}
  Since $\eta(Z_1(\eta))<\infty$ almost surely, the decreasing
  restrictions $\eta_{Z_n}$ stabilize: almost surely there exists a
  finite random integer $N$ such that $\eta_{Z_n}=\eta_Z$ for all
  $n\geq N$.  To verify the measurably closed convergence, let $X_n$
  be $\sigma(\eta_{Z_n})$-measurable real random variables with
  pointwise limit $X$. Write $X_n=g_n(\eta_{Z_n})$ for measurable real
  functions $g_n$. Define $g(\nu)=\liminf_n g_n(\nu)$ when this limit
  inferior is finite, and $g(\nu)=0$ otherwise. Eventual equality of
  the restrictions gives $X=g(\eta_Z)$ almost surely.  The assertion
  follows from Proposition ~\ref{t:spatMarkov3}.
\end{proof}

\section{An example on the measurability properties 
  of   $\lambda_Z$}
\label{sec:counter}

In this section we provide an example of a graph-limit $Z$ of bounded
stopping sets such that $\lambda_Z$ is not measurable with respect to
the $\BP$-completion of $\sigma(\eta_Z)$.  Therefore, and since $Z$ is
finitely insertion tolerant, the conditioning with respect to
$\lambda_Z$ in \eqref{eMarkov2} cannot be dropped under the assumption
of Theorem~\ref{t:spatMarkov2}.

In this section we assume that 
\begin{equation*}
  \BX:=\{0,1\}^*=\bigcup_{n=0}^\infty\{0,1\}^n
\end{equation*}
is the rooted binary tree, where $\{0,1\}^0:=\{o\}$ contains only the
root $o$ (the empty word).  If $v\in \{0,1\}^n$ we write
$|v|=n$. Every $v\in\BX$ with $|v|=n$ can be written as
$v=(o,v_1,\ldots,v_n)$ and has the two children $(o,v_1,\ldots,v_n,0)$
and $(o,v_1,\ldots,v_n,1)$, denoted $v0$ and $v1$ respectively.  We
use $B_n:=\{v\in\BX:|v|\leq n\}$, $n\in\N_0$, as a localizing
sequence.  Thus the bounded subsets of $\BX$ are precisely the finite
sets.

Let $\eta$ be a Poisson process on $\BX$ with intensity measure
\begin{equation*}
  \lambda:=t\sum_{v\in\BX}\delta_v,
  \qquad \text{for}\qquad 
  t:=\log(15/8).
\end{equation*}
The random variables $\eta\{v\}$, $v\in\BX$, are independent and
Poisson distributed with mean $t$.  For $\mu\in\bN$ and $v\in\BX$, let
$v\in Z(\mu)$ if one of the children of $v$ starts an infinite
$\mu$-empty descending ray.  In other words, $v\in Z(\mu)$ if and only
if there exist $v_0,v_1,v_2,\ldots$ such that $v_0=v$,
$v_{k+1}\in\{v_k0,v_k1\}$ for $k\ge 0$ and $\mu\{v_k\}=0$ for
$k\geq1$.  Notice in particular that the value of $\mu\{v\}$ plays no
role in deciding whether $v\in Z(\mu)$.

This construction has a simple percolation interpretation.  Declaring
a vertex $v$ \emph{open} when $\eta\{v\}=0$ yields Bernoulli site
percolation on the binary tree with parameter
\begin{equation*}
  q:=\BP(\eta\{v\}=0)=e^{-t}=8/15> 1/2.
\end{equation*}
The supercritical regime $q>\frac12$ is essential here, since only
then infinite open descending rays occur with positive probability and
hence $Z(\eta)$ is non-trivial, see \cite[Chapter~5]{MR3616205}. 
In fact, $Z(\eta)$ can be viewed as
the one-step enlargement, towards the root, of the set of vertices
starting infinite open descending rays.  Since open vertices carry no
points of $\eta$, this infinite open structure is itself invisible in
$\eta_Z$. The latter only records the occupied vertices of $Z(\eta)$,
which may be regarded as points on the boundary, towards the root, of
the infinite open structure, together with their Poisson
multiplicities.  Thus $\eta_Z\{v\}=0$ has two possible explanations:
either $v\notin Z(\eta)$, or $v\in Z(\eta)$ but $v$ itself is open.
The point of the following proposition is that this ambiguity cannot
be resolved even by observing the whole process $\eta_Z$.  The
particular value $q=8/15$ is chosen to obtain convenient constants in
the quantitative argument establishing this fact.

\begin{proposition}\label{prop:tree-counterexample}
  The graph-measurable mapping $Z$ is the graph-limit of a sequence of
  bounded stopping sets and $\lambda_Z(\{o\})=t\,\I\{o\in Z(\eta)\}$
  is not measurable with respect to the $\BP$-completion of
  $\sigma(\eta_Z)$.
\end{proposition}
\begin{proof}
  We first verify the graph-limit assertion.  For $\mu\in\bN$ and
  $n\in\N$ define $Z_n(\mu)$ to be the set of all $v\in B_n$ such that
  there is a descendant $w$ of $v$ with $|w|=n$ such that $\mu\{u\}=0$
  for every $u\in(v,w]$.  All vertices at generation $n$ belong to
  $Z_n(\mu)$, and for $|v|<n$ we have the recursion
  \begin{displaymath}
    \I\{v\in Z_n(\mu)\}=
    \bigvee_{j\in\{0,1\}}
    \left(
      \I\{vj\in Z_n(\mu)\}\I\{\mu\{vj\}=0\}
    \right).
  \end{displaymath}
  It follows by induction, starting at generation $n$, that the
  queried set $Z_n(\mu)$ is determined by the restriction of $\mu$ to
  $Z_n(\mu)$.  More precisely, if $\nu$ agrees with $\mu$ on
  $Z_n(\mu)$, then the same induction gives $Z_n(\nu)=Z_n(\mu)$.
  Hence $Z_n$ is a stopping set.  It is bounded since
  $Z_n(\mu)\subset B_n$.

  For fixed $v\in\BX$ and $\mu\in\bN$, the indicators
  $\I\{v\in Z_n(\mu)\}$ are decreasing once $n\geq |v|$.  By
  definition of the binary tree, their limit equals $1$ if and only if
  one of the children of $v$ starts an infinite $\mu$-empty descending
  ray.  
  Consequently,
  \begin{displaymath}
    \lim_{n\to\infty}\I\{v\in Z_n(\mu)\}
    =\I\{v\in Z(\mu)\},\quad v\in\BX,\, \mu\in\bN.
  \end{displaymath}
  Thus $Z$ is the graph limit of the bounded stopping sets $Z_n$,
  hence is graph-measurable.

  We turn to the measurability assertion.  Put
  \begin{displaymath}
    I_v:=\I\{v\in Z(\eta)\},
    \qquad
    J_v:=\I\{\eta\{v\}=0\}I_v,
    \qquad
    S_v:=\eta\{v\}I_v.
  \end{displaymath}
  Thus $J_v=1$ means that $v$ is open and starts an infinite open
  descending ray.  By the definition of $Z$,
  \begin{equation}
    \label{4}
    I_v=J_{v0}\vee J_{v1},
    \qquad
    J_v=\I\{\eta\{v\}=0\}(J_{v0}\vee J_{v1}).
  \end{equation}
  Moreover,
  \begin{equation*}
    \eta_Z=\sum_{v\in\BX}S_v\delta_v,
  \end{equation*}
  and hence
  \begin{equation}
    \label{5}
    \sigma(\eta_Z)=\sigma(S_v:v\in\BX). 
  \end{equation}
  Let $r:=\BP(J_v=1)$.  The two descendant subtrees of $v$ are
  independent, and $\eta\{v\}$ is independent of both of them.
  Therefore, \eqref{4} gives $r=q\bigl(1-(1-r)^2\bigr)$.  Since we are
  in the supercritical case $q=8/15>1/2$ we can exclude the trivial
  solution $r=0$.
  The relevant non-zero solution is
  \begin{equation}
    \label{6}
    r=2-\frac1q=\frac18. 
  \end{equation}
  It follows also that for all $v\in\BX$
  \begin{equation}
    \label{7}
    \BP(I_v=1)=1-(1-r)^2=\frac{15}{64},
    \qquad
    \BP(I_v=0)=\frac{49}{64}. 
  \end{equation}

  For $n\in\N_0$, let $\mathcal F_n:=\sigma(S_v:|v|\leq n)$.  These
  $\sigma$-fields increase to $\sigma(\eta_Z)$ by \eqref{5}.  Set
  $\theta_n:=\BP(J_o=1\mid\mathcal F_n)$.  If $J_o$ were measurable
  with respect to the $\BP$-completion of $\sigma(\eta_Z)$, martingale
  convergence would 
  give $\theta_n\to J_o$ in $L^2$, so that
  \begin{equation}
    \label{eq:theta18}
    \lim_{n\to\infty}\BE[\theta_n^2]
    =\BE[J_o^2]=\BP(J_o=1)=\frac18.
  \end{equation}
  We will now derive an upper bound on $\BE[\theta_n^2]$ which
  contradicts the above. 

  Write $S^{(n)}:=(S_v)_{v\in B_n}$.  For each possible value
  $s=(s_v)_{v\in B_n}$ of $S^{(n)}$, define
  \begin{displaymath}
    a_n(s):=\BP(J_o=0,S^{(n)}=s),
    \qquad
    b_n(s):=\BP(J_o=1,S^{(n)}=s). 
  \end{displaymath}
  By \eqref{6},
  \begin{equation}\label{11}
    \sum_s a_n(s)=\frac78,
    \qquad
    \sum_s b_n(s)=\frac18.
  \end{equation}
We also note that
\begin{equation*}
    a_0(0)=\BP(J_o=0,S_o=0)=\BP(I_o=0)
    =\frac{49}{64}.
  \end{equation*}
  Moreover, if $s_o>0$ we have $b_n(s)=0$, since $J_o=1$ implies
  $S_o=0$.  It follows by induction combined with the identities
  \eqref{15} and \eqref{14} below, that $b_n(s)>0$ implies $a_n(s)>0$,
  provided that $s_o=0$. Otherwise we have $b_n(s)=0$.  We may
  therefore define
  \begin{displaymath}
    Q_n:=\sum_{s\in \N_0^{B_n}}\frac{b_n(s)^2}{a_n(s)}, 
  \end{displaymath}
  where terms with $b_n(s)=0$ are taken to be zero.

  We next derive a recursion for $Q_n$.  A realization of $S^{(n+1)}$
  with zero at the root can be written as $(0,s_0,s_1)$, where $s_0$
  and $s_1$ are the corresponding observations in the two descendant
  subtrees.  Notice that
  \begin{equation*}
 \{J_o=0,S_o=0\}=\{I_o=0\}
 =\{J_{o0}=0,J_{o1}=0\}.
  \end{equation*}
  Indeed, if $I_o=1$ and $J_o=0$, then $\eta\{o\}>0$ and hence
  $S_o>0$.  The second equality follows from the first relation in
  \eqref{4}.  Independence of the two descendant subtrees therefore
  gives
  \begin{equation}
    \label{14}
    a_{n+1}(0,s_0,s_1)=a_n(s_0)a_n(s_1).
  \end{equation}
  If $J_o=1$, then $\eta\{o\}=0$ and the pair $(J_{o0},J_{o1})$ can
  take the three values $(1,0),(0,1),(1,1)$.  Hence
  \begin{equation}
    \label{15}
    b_{n+1}(0,s_0,s_1)=q(b_n(s_0)a_n(s_1)
    +a_n(s_0)b_n(s_1)+b_n(s_0)b_n(s_1)).
  \end{equation}
  As noticed above, observations with a positive value at the root
  make no contribution to $Q_{n+1}$. 
  Consequently,
  \begin{align*}
    Q_{n+1}
    &=q^2\sum_{s_0,s_1\in \N_0^{B_n}} \frac{\bigl(
      b_n(s_0)a_n(s_1)+a_n(s_0)b_n(s_1)
      +b_n(s_0)b_n(s_1)\bigr)^2}{a_n(s_0)a_n(s_1)}\\
    &=q^2\bigg(2\Big(\sum_s a_n(s)\Big)Q_n
      +Q_n^2+2\Big(\sum_s b_n(s)\Big)^2
      +4\Big(\sum_s b_n(s)\Big)Q_n \bigg).
  \end{align*}
  Using \eqref{11} and $q=8/15$, we obtain
  \begin{align}\label{16}
    Q_{n+1}=q^2\bigg(2\frac78 Q_n+Q_n^2
    +2\left(\frac18\right)^2+4\frac18 Q_n\bigg)
    =\frac2{225}+\frac{16}{25}Q_n+\frac{64}{225}Q_n^2. 
  \end{align}

  At generation zero, $b_0(s)$ can be positive only for $s=0$.  By
  \eqref{6} and \eqref{7}, $b_0(0)=1/8$.  Consequently,
  \begin{displaymath}
    Q_0=\frac{(1/8)^2}{49/64}=\frac1{49}. 
  \end{displaymath}
  The polynomial
  \begin{equation*}
    f(x):=\frac2{225}+\frac{16}{25}x+\frac{64}{225}x^2
  \end{equation*}
  is increasing on $\R_+$ and satisfies $f(1/16)=1/20<1/16$.  Since
  $Q_0<1/16$, \eqref{16} gives by induction that
  $Q_n\leq 1/16$ for all $n\in\N_0$.

  By the definitions of $a_n$ and $b_n$,
  \begin{equation}
    \BE[\theta_n^2]
    =\sum_s\frac{b_n(s)^2}{a_n(s)+b_n(s)}
    \leq Q_n\leq\frac1{16}.
  \end{equation}
  which contradicts \eqref{eq:theta18}.  Thus $J_o$ is not measurable
  with respect to the $\BP$-completion of $\sigma(\eta_Z)$.  Finally,
  $J_o=I_o\I\{S_o=0\}$.  Indeed, if $I_o=1$, then $S_o=\eta\{o\}$,
  while both sides vanish if $I_o=0$. Since $S_o=\eta_Z\{o\}$ is
  $\sigma(\eta_Z)$-measurable, measurability of $I_o$ with respect to
  the completed $\sigma$-algebra would imply the same for $J_o$.
  Thus, $I_o$, and hence $\lambda_Z(\{o\})=tI_o$, is not measurable
  with respect to the $\BP$-completion of $\sigma(\eta_Z)$.
\end{proof}

\section{Some distributional identities for stopping sets}\label{sec:identities}

In this section we fix a stopping set $Z$ and study the
distribution of $(\eta_Z,\lambda_{Z^c},\eta_{Z^c})$ on the event
$\{\eta(Z(\eta))<\infty\}$.

Let $C_n$ be the set of all $\bx\in\BX^n$ such that
$x_1,\ldots,x_n\in Z(\delta_{\bx})$. For instance, in the convex hull
setting, $C_n$ is the collection of $n$-tuples of points which are
vertices of their convex hull, with multiplicities retained.

Given $B\in\cX$ the (joint) distribution of
$\big(\eta_Z,\lambda_{Z^c}(B),\eta_{Z^c}(B)\big)$ on the event
$\{\eta_Z(\BX)<\infty\}$ can be described by a measure $\lambda^*_B$
on $\bN\times[0,\infty]$ which is defined as
\begin{displaymath}
  \lambda^*_B(\cdot):=e^{-\lambda(Z(0))}\delta_{(0,\lambda(B\cap Z(0)^c))}
  +\sum^\infty_{n=1}\frac{1}{n!}
  \int_{C_n}\I\big\{\big(\delta_{\bx},\lambda(B\cap [\bx])\big)
  \in\cdot\big\}e^{-\lambda([\bx]^c)}\, \lambda^n(d\bx).
\end{displaymath}

In the following theorem, for $t<\infty$ we set $t^\infty/\infty!:=0$.
For $t=\infty$ we set $e^{-t} t^m/m!:=0$ for $m\in\N_0$ and
$e^{-t}t^m/m!:=1$ for $m=\infty$.
Define $\bar{\N}_0:=\N_0\cup\{\infty\}$.
Throughout, we use the convention $0^0=1$.

\begin{theorem}\label{t1}
  Let $f\colon \bN\times[0,\infty]\times\bar{\N}_0\to\R_+$ be a
  measurable function and $B\in\cX$.  Then
  \begin{align}
    \label{e:2.4a}
    \BE \big[\I\{\eta_Z(\BX)<\infty\}
    f\big(\eta_Z,\lambda_{Z^c}(B),\eta_{Z^c}(B)\big)\big]
    =\sum_{m\in\bar\N_0}\int \frac{t^m}{m!} e^{-t}f(\mu,t,m)
    \,\lambda^*_B(d(\mu,t)).
  \end{align} 
\end{theorem}
\begin{proof} Let $n\in\N$. By Corollary~\ref{c:stoppingMarkov},
  \begin{align*}
    I_n&:=\BE\big[
         \I\{\eta_Z(\BX)=n\}
         f\big(\eta_Z,\lambda_{Z^c}(B),\eta_{Z^c}(B)\big)\big]\\
    &=\BE \Big[\I\{\eta(Z(\eta))=n\}
      \BE\big[f\big(\eta_{Z},\lambda(B\cap Z(\eta)^c),
      \eta\big(B\cap Z(\eta)^c)\big)\mid \eta_Z\big]\Big]\\
    &=\BE\sum_{m\in\bar{\N}_0} \I\{\eta(Z(\eta))=n\}
      f\big(\eta_{Z},\lambda(B\cap Z(\eta)^c),m\big)
      \frac{\lambda(B\cap Z(\eta)^c)^m}{m!}
      \exp\big[\!-\lambda(B\cap Z(\eta)^c)\big].
  \end{align*}
  By \cite[Proposition A.5]{LaPeYo23},
  \begin{displaymath}
    I_n=\sum_{m\in\bar{\N}_0}\frac{1}{n!}
    \int_{C_n}f\big(\delta_{\bx},\lambda(B\cap [\bx]),m\big)
    \frac{\lambda(B\cap [\bx])^m}{m!}
    \exp\big[-\lambda(B\cap [\bx])\big]
    \exp\big[-\lambda([\bx]^c)\big]\,\lambda^n(d\bx).
  \end{displaymath}
  Similarly, we obtain
  \begin{align*}
    \BE &\big[\I\{\eta_Z(\BX)=0\}
          f\big(\eta_Z,\lambda_{Z^c}(B),\eta_{Z^c}(B)\big)\big]\\
        &=\sum_{m\in\bar{\N}_0}f\big(0,\lambda(B\cap Z(0)^c),m\big)
          \frac{\lambda(B\cap Z(0)^c)^m}{m!}
          \exp\big[\!-\lambda(B\cap Z(0)^c)\big]
          \exp\big[\!-\lambda(Z(0))\big].
  \end{align*}
  Taking the sum over $n\in\N_0$ concludes the proof.
\end{proof}

\begin{remark}
  Assume that $\lambda(B)<\infty$. Then 
  \eqref{e:2.4a} simplifies to
  \begin{align}
    \label{e:2.4b}
    \BE \big[\I\{\eta_Z(\BX)<\infty\}
    f\big(\eta_Z,\lambda_{Z^c}(B),\eta_{Z^c}(B)\big)\big]
    =\sum^\infty_{m=0}\int \frac{t^m}{m!} e^{-t}f(\mu,t,m)\,\lambda^*_B(d(\mu,t)).
  \end{align}
  If $\lambda(\BX)=\infty$, then $\BP(\eta(\BX)=\infty)=1$,
  and \eqref{e:2.4a} simplifies in the case $B=\BX$ to
  \begin{displaymath}
    \BE\big[\I\{\eta_Z(\BX)<\infty\}
    f\big(\eta_Z,\lambda(Z(\eta)^c),\eta_{Z^c}(\BX)\big)\big]
    =\int f(\mu,t,\infty)\,\lambda^*_{\BX}(d(\mu,t)).
  \end{displaymath}
\end{remark}

For $m\in\N_0$ and $k\in\N_0$, we denote by
$ m^{(k)}:=m\cdots (m-k+1)$, the $k$-th descending factorial of $m$.
In the case $\lambda(B)<\infty$ and for special choices of the
function $f$ we obtain the following corollaries.

\begin{corollary}\label{c2}
  Let $B\in\cX_0$ and let $f\colon \bN\times\R_+\to\R_+$ be measurable
  and $k\in\N_0$.  Then
  \begin{displaymath}
    \BE \big[\I\{\eta_Z(\BX)<\infty\}
    f\big(\eta_Z,\lambda_{Z^c}(B)\big)\eta_{Z^c}(B)^{(k)}\big]
    =\int f(\mu,t)t^k\,\lambda^*_B(d(\mu,t)).
  \end{displaymath}
\end{corollary}
\begin{proof}
  By \eqref{e:2.4b}, 
  \begin{align*}
    \BE \big[\I\{\eta_Z(\BX)<\infty\}
    f\big(\eta_Z,\lambda_{Z^c}(B)\big)\eta_{Z^c}(B)^{(k)}\big]
    &=\sum^\infty_{m=k}\int \frac{t^m}{m!}e^{-t} \frac{m!}{(m-k)!}
    f(\mu,t)\,\lambda^*_B(d(\mu,t)),
  \end{align*}
  which implies the result. 
\end{proof}

\begin{corollary}  \label{c3}
  Let $B\in\cX_0$.  Let $f\colon \bN\times\R_+\to\R_+$ be measurable
  and let $z\in\R_+$. Then
 \begin{displaymath}
    \BE \Big[\I\{\eta_Z(\BX)<\infty\}
    f\big(\eta_Z,\lambda_{Z^c}(B)\big)
    e^{-z\eta_{Z^c}(B)}\Big]
    =   \int f(\mu,t)e^{-t(1-e^{-z})}\,\lambda^*_B(d(\mu,t)).
  \end{displaymath}
\end{corollary}
\begin{proof} By \eqref{e:2.4b},
  \begin{align*}
    \BE \Big[\I\{\eta_Z(\BX)<\infty\}
    f\big(\eta_Z,\lambda_{Z^c}(B)\big)
    e^{-z\eta_{Z^c}(B)}\Big]
    &=\int f(\mu,t)\sum_{m=0}^\infty
      e^{-t}\frac{t^m}{m!}e^{-zm}
      \,\lambda^*_B(d(\mu,t))  \\
    &=\int f(\mu,t) e^{-t}e^{te^{-z}}
    \,\lambda^*_B(d(\mu,t))  \\
    &= \int f(\mu,t) e^{-t(1-e^{-z})}\,\lambda^*_B(d(\mu,t)). \qedhere
  \end{align*}
\end{proof}

Taking $z=0$ in Corollary~\ref{c3} (or $k=0$ in Corollary~\ref{c2}) and 
writing out  the definition of $\lambda^*_B$ yields the following result.

\begin{corollary}  \label{c5}
  Let $B\in\cX_0$ and $f\colon \bN\times\R_+\to\R_+$ be measurable. Then
 \begin{align*}
\BE& \big[\I\{\eta_Z(\BX)<\infty\}f\big(\eta_Z,\lambda_{Z^c}(B)\big)\big]\\
&=f(0,\lambda(Z(0)^c\cap B))e^{-\lambda(Z(0))}  +\sum^\infty_{n=1}\frac{1}{n!}
  \int_{C_n}f\big(\delta_{\bx},\lambda(B\cap [\bx])\big)e^{-\lambda([\bx]^c)}\, \lambda^n(d\bx).
  \end{align*}
\end{corollary}

\section{Distributional identities for Markov sets}
\label{sec:markov-sets}

In this section we will work with a Markov set $Z$,
starting with the following simple fact.

\begin{lemma}\label{l:fi}
Suppose that $Z$ is a Markov set and let $B\in\cX$. Then
  $\BP(\lambda_{Z^c}(B)<\infty)=\BP(\eta_{Z^c}(B)<\infty)$.
\end{lemma}
\begin{proof}
By the spatial Markov property \eqref{eMarkov2},
  \begin{align*}
    \BP(\eta_{Z^c}(B)<\infty)
    &=
    \BE\big[
      \BP\big(\eta_{Z^c}(B)<\infty\mid \eta_Z,\lambda_Z\big)\big]  \\
    &=
    \BE\big[\Pi_{\lambda_{Z^c}}\big(\mu\in\bN:\mu(B)<\infty\big)\big]  
    =\BE \I\{\lambda_{Z^c}(B)<\infty\}.
  \end{align*}
  Indeed, under $\Pi_{\lambda_{Z^c}}$, the random variable $\mu(B)$
  is Poisson with parameter $\lambda_{Z^c}(B)$, and is finite almost
  surely if and only if this parameter is finite.
\end{proof}

We continue with a symmetry property of the distribution of
$\big(\lambda_{Z^c}(B),\eta_{Z^c}(B)\big)$

\begin{proposition}  \label{c2.3}
  Assume that $Z$ is a Markov set and let $B\in\cX$
  such that $\BP(\lambda_{Z^c}(B)<\infty)=1$. Let $k,l\in\N_0$. Then
  \begin{align} \label{e:2.9}
    \BE \big[\eta_{Z^c}(B)^{(k)}\lambda_{Z^c}(B)^l\big]
    =\BE \lambda_{Z^c}(B)^{k+l}=\BE \eta_{Z^c}(B)^{(k+l)}.
  \end{align}
\end{proposition}
\begin{proof}
  By Lemma~\ref{l:fi}, $\BP(\eta_{Z^c}(B)<\infty)=1$. By the spatial
  Markov property \eqref{eMarkov2} and the factorial moments of a Poisson
  distribution,
  \begin{align*}
    \BE \big[\eta_{Z^c}(B)^{(k)}\lambda_{Z^c}(B)^l\big]
    &=
    \BE\Big[
      \lambda_{Z^c}(B)^l
      \BE\big[\eta_{Z^c}(B)^{(k)}\mid \eta_Z,\lambda_Z\big]
    \Big]  \\
    &=
    \BE\big[\lambda_{Z^c}(B)^l\lambda_{Z^c}(B)^k\big]  
    =
    \BE\lambda_{Z^c}(B)^{k+l}.
  \end{align*}
  This proves the first identity in \eqref{e:2.9}. In particular we
  obtain for $l=0$ that
  \begin{equation}
    \label{e:momentshull}
    \BE \eta_{Z^c}(B)^{(k)}=\BE\lambda_{Z^c}(B)^k,
  \end{equation}
  proving the second identity in \eqref{e:2.9}.
\end{proof}

The Efron formula \eqref{e:Efron} is a very special case of the
following result. The Mecke equation
would be an alternative direct way to derive this result, see the
proof of the forthcoming Theorem~\ref{t:2.7} in the case $k=1$.

\begin{corollary}
  \label{c2.5}
  Let $Z$ be a Markov set. Then
  \begin{equation}\label{e:Efrongeneral}
    \BE \eta_Z(B)=\BE \lambda_Z(B),\quad B\in\cX.
  \end{equation}
\end{corollary} 
\begin{proof}
  By monotone convergence, it suffices to assume $\lambda(B)<\infty$.
  By \eqref{e:momentshull},
  \begin{displaymath}
    \BE \eta_Z(B)
    =\BE\eta(B)-\BE\eta_{Z^c}(B)
    =\lambda(B)-\BE\lambda_{Z^c}(B)
    =\BE \lambda_Z(B). \qedhere
  \end{displaymath}
\end{proof}

\begin{remark}\label{r:signedsecond}
  Suppose that $Z$ is a Markov set.  Corollary~\ref{c2.5} states the
  point process $\eta_Z$ and the random measure $\lambda_Z$ have the
  same intensity measure. For $f\in L^1(\lambda)\cap L^2(\lambda)$ set
  $I_Z(f):=\eta_Z(f)-\lambda_Z(f)$.  It then follows from the Markov
  property and the conditional variance formula that
  \begin{displaymath}
    \BE I_Z(f)^2=\int f(x)^2 \BP(x\in Z(\eta))\,\lambda(dx),
  \end{displaymath}
  see the proof of \cite[Theorem~5.1]{LM23}. Alternatively, this can
  be derived from \eqref{e:momentshull} (for $k\in\{1,2\}$) with a
  simple calculation. More generally it follows by polarization that
  $I_Z$ satisfies the isometry formula
  \begin{equation}\label{e:isometry}
    \BE I_Z(f)I_Z(g)
    =\int f(x)g(x) \BP(x\in Z(\eta))\,\lambda(dx),
    \quad f,g\in L^1(\lambda)\cap L^2(\lambda).
  \end{equation}
\end{remark}

\begin{remark}
  In the convex hull case the formula \eqref{e:momentshull} for
  $B=\R^d$ becomes
  \begin{equation}
    \label{e:713}
    \BE \lambda(Z(\eta)^c)^{k}= \BE \eta_{Z^c}(\R^d)^{(k)},
  \end{equation}
  which is the Poisson version of equation (2.4) in \cite{Buchta05}.  In
  fact, the latter implies \eqref{e:713} via a simple
  computation. Since in our case
  $\eta_Z(\BX)+\eta_{Z^c}(\BX)=\eta(\BX)$ is not deterministic, it
  does not seem to be possible to derive from \eqref{e:713} a Poisson
  counterpart of \cite[Theorem~1]{Buchta05} or of
  \cite[Theorem~1]{Buchta23}.
\end{remark}

Next, we compress \eqref{e:2.9} in just one formula.

\begin{theorem}
  \label{t:jointtransfrom}
  Assume that $Z$ is a Markov set, and let $B\in\cX$ such that
  $\BP(\lambda_{Z^c}(B)<\infty)=1$.  Then, for all $s\ge -1$ and
  $t\in\R$,
  \begin{equation}\label{e:transform}
    \BE \big[(s+1)^{\eta_{Z^c}(B)}e^{t\lambda_{Z^c}(B)}\big]
    =\BE e^{(s+t)\lambda_{Z^c}(B)}.
  \end{equation}
\end{theorem}
\begin{proof}
  Using the spatial Markov property  we have
  \begin{align*}
    \BE \big[(s+1)^{\eta_{Z^c}(B)}e^{t\lambda_{Z^c}(B)}\big]
    &=\BE \Big[e^{t\lambda_{Z^c}(B)}\BE\big[(s+1)^{\eta_{Z^c}(B)}\mid \eta_Z,\lambda_Z\big]\Big]\\
    &=\BE \Big[e^{t\lambda_{Z^c}(B)}\int (s+1)^{\mu(B)}\Pi_{\lambda_{Z(\eta)^c}}(d\mu)\Big]\\
    &=\BE \big[e^{t\lambda_{Z^c}(B)}e^{s\lambda_{Z^c}(B)}\big]. \qedhere
  \end{align*}
\end{proof}

\begin{remark}
  Note that the choice of $t=0$ turns \eqref{e:transform} into
  Theorem~1 from \cite{BeReitz15}. The results of this and the
  previous section generalize Theorems~1 and~2 in \cite{BeReitz15} in
  several ways.  First, they apply to general stopping sets and in
  particular to hull operators on general spaces, not just to the
  convex hull in Euclidean space. Second, they cover joint
  distributions and not just the marginal distributions of
  $\eta_{Z^c}(B)$ and $\lambda_{Z^c}(B)$.
  Finally, our results do not require any further properties of
  $\lambda$ and even allow for the case $\lambda(\BX)=\infty$.
\end{remark}

While Theorem~\ref{t:jointtransfrom} describes a general Cox process
property, we shall now use the fact that $\eta_Z+\eta_{Z^c}$ is a
Poisson process. The next result generalizes
\cite[Equation~(12)]{Zuyev99}.

\begin{theorem}
  Suppose that $Z$ is a Markov set and let $f\colon \BX\to\R_+$ be
  measurable with $\lambda(1-e^{-f})<\infty$.  Then
  \begin{align}
    \label{e:23.16a}
    \BE \big[\exp[-\eta_Z(f)+\lambda_Z(1-e^{-f})]\big]=1.
  \end{align}
\end{theorem}
\begin{proof}
  Put $g:=1-e^{-f}$. By the spatial Markov property,
  \begin{align*}
    \BE e^{-\eta(f)}
    &=
    \BE\big[e^{-\eta_Z(f)}e^{-\eta_{Z^c}(f)}\big]  \\
    &=
    \BE\Big[
      e^{-\eta_Z(f)}
      \BE\big[e^{-\eta_{Z^c}(f)}\mid \eta_Z,\lambda_Z\big]
    \Big]  
    =
    \BE\big[e^{-\eta_Z(f)}e^{-\lambda_{Z^c}(g)}\big].
  \end{align*}
  Since $\eta$ is Poisson, $\BE e^{-\eta(f)}=e^{-\lambda(g)}$.
  Therefore,
  \begin{equation}\label{e:23.16b}
    \BE\big[e^{-\eta_Z(f)}e^{-\lambda_{Z^c}(g)}\big]=e^{-\lambda(g)}.
  \end{equation}
  Since $\lambda(g)<\infty$ and
  $\lambda(g)=\lambda_Z(g)+\lambda_{Z^c}(g)$, multiplying by
  $e^{\lambda(g)}$ gives
  \begin{displaymath}
    \BE\big[e^{-\eta_Z(f)+\lambda_Z(g)}\big]=1. \qedhere
  \end{displaymath}
\end{proof}

Below we set $\infty^{(0)}:=1$, $\infty^{(k)}:=\infty$ for $k\geq1$, and
    $0\cdot\infty:=0$.

\begin{corollary}\label{c:23.11}  
  Suppose that $Z\colon \bN\to\cX$ is a Markov set. 
  Let $B\in\cX$ and $k\in\N$. Then
  \begin{equation} \label{e:2.11a}
    \sum^k_{i=0}\binom{k}{i}\,\BE \big[\eta_Z(B)^{(i)}
    \lambda_{Z^c}(B)^{k-i}\big]=\lambda(B)^k.
  \end{equation}
  If $B\in\cX_0$, then we have moreover,
  \begin{align} \label{e:2.11b}
    \sum^k_{i=0}(-1)^i\binom{k}{i}\,\BE
    \big[\eta_Z(B)^{(i)}\lambda_{Z}(B)^{k-i}\big]=0.
  \end{align}
\end{corollary}
\begin{proof}
  Assume first that $B\in\cX_0$, so that $\lambda(B)<\infty$. Then
  $\lambda_Z(B)$, $\lambda_{Z^c}(B)$ are both finite and
  $\eta_Z(B)\le\eta(B)$ has finite exponential moments.  Taking
  $f:=t\I_B$ in \eqref{e:23.16b}, we obtain
  \begin{displaymath}
    \BE\big[\exp[-t\eta_Z(B)-(1-e^{-t})\lambda_{Z^c}(B)]\big]
    = \exp[-(1-e^{-t})\lambda(B)].
  \end{displaymath}
  Put $r:=1-e^{-t}$. Then
  \begin{displaymath}
    e^{-t\eta_Z(B)}
    =
    (1-r)^{\eta_Z(B)}
    =
    \sum_{m=0}^\infty
      \frac{(-1)^m r^m}{m!}\eta_Z(B)^{(m)}.
  \end{displaymath}
  Expanding also $e^{-r\lambda_{Z^c}(B)}$ and comparing coefficients
  yields
  \begin{displaymath}
    \sum_{i=0}^k
    \binom{k}{i}
    \BE\big[\eta_Z(B)^{(i)}\lambda_{Z^c}(B)^{k-i}\big]
    = \lambda(B)^k.
  \end{displaymath}
  This proves \eqref{e:2.11a} for $B\in\cX_0$. The extension to
  arbitrary $B\in\cX$ follows by applying the identity to
  $B\cap B_n$, where $B_n\uparrow\BX$ is a localising sequence, and
  using monotone convergence.
  Finally, starting instead from \eqref{e:23.16a} with $f=t\I_B$ gives
  \begin{displaymath}
    \BE\big[(1-r)^{\eta_Z(B)}e^{r\lambda_Z(B)}\big]=1.
  \end{displaymath}
  Comparing coefficients gives, for $k\ge1$,
  \begin{displaymath}
    \sum_{i=0}^k
    (-1)^i\binom{k}{i}
    \BE\big[\eta_Z(B)^{(i)}\lambda_Z(B)^{k-i}\big]
    = 0. \qedhere
  \end{displaymath}
\end{proof}

Equation \eqref{e:2.11a} has been noticed in
\cite[Example~5.10]{LM23}, while equation \eqref{e:2.11b} comes from
\cite{Zuyev99}.  Next we extend the conclusion of Example~6.3 from
\cite{LM23} to Markov sets and provide its localized variant.

\begin{corollary}
  Suppose that $Z$ is a Markov set and let $B\in\cX_0$. Then, for all $k\in\N$,
  \begin{align}\label{e:2.15}
    \lambda(B)^k
    &=\BE  \lambda_{Z^c}(B)^k+\BE \eta_Z(B)^{(k)}\\ 
    &\quad +\sum^{k-1}_{i=1}\binom{k}{i}
      \BE \int_{B^{k-i}}\lambda_B(Z(\eta+\delta_{\bx})^c)^i \,
      \I\big\{x_1,\ldots,x_{k-i}\in Z(\eta+\delta_{\bx})\big\}
      \,\lambda^{k-i}(d\bx).\notag      
  \end{align}
\end{corollary}
\begin{proof} For each $i\in[k-1]$ we have that
  \begin{displaymath}
    \eta_Z(B)^{(k-i)}=\int\I\big\{x_1,\ldots,x_{k-i}\in B
    \cap Z(\eta)\big\}\,\eta^{(k-i)}(d\bx),
  \end{displaymath}
where $\eta^{(k-i)}$ denotes the $(k-i)$-th factorial measure
of $\eta$. Therefore,
equation \eqref{e:2.15} follows from \eqref{e:2.11a}  and the multivariate
Mecke equation, see \cite[Theorem~4.4]{LastPenrose17}.
\end{proof}

\section{Identities based on thinning decomposition}
\label{sec:ident-based-thinn}

The next result generalizes Theorem~4 of \cite{BeReitz15}.

\begin{theorem}[Generating function of the revealed points]
  \label{thr:BR-main}
  Assume that $\lambda(\BX)<\infty$.  Let $Z$ be a
stopping set. Then, for every $x\in[0,1]$ and $s\ge -1$,
  \begin{equation}\label{eq:BR-general}
    \BE x^{\eta(Z(\eta))}
    = e^{-sx\lambda(\BX)}
    \BE\left[
      (s+1)^{\eta^x(Z(\eta^x)^c)}
      \exp\left\{(x-1+sx)\lambda(Z(\eta^x))\right\}
    \right],
  \end{equation}
  where $\eta^x$ is a Poisson point process of intensity $x\lambda$.
\end{theorem}
\begin{proof}
  Let $x\in[0,1]$.  We first prove a simpler result saying that
  \begin{equation}
    \label{eq:BR-main}
    \BE x^{\eta(Z(\eta))}= \BE\exp[(x-1)\lambda(Z(\eta^x))].
  \end{equation}
  For $x=0$ we have $\BP(\eta^x=0)=1$ and \eqref{eq:BR-main} boils
  down to $\BP(\eta(Z(\eta))=0)=e^{-\lambda(Z(0))}$, which is a direct
  consequence of Corollary~\ref{c5}.

  We let $\eta^x$ be an independent $x$-thinning of $\eta$,
  constructed by independently marking every point of $\eta$ with a
  Bernoulli random variable.  A point is retained in $\eta^x$ with
  probability $x$ and deleted into $\eta-\eta^x$, otherwise. Then
  $\eta^x$ and $\eta-\eta^x$ are independent Poisson processes with
  intensity measure $x\lambda$ and $(1-x)\lambda$, respectively; see
  e.g.\ \cite[Corollary 5.9]{LastPenrose17}.
  Denote $Z_x:=Z(\eta^x)$ and $N:=\eta(Z(\eta))$.  Conditionally on
  $\eta$, the event $\eta_{Z(\eta)}\leq\eta^x$ means that every point of
  $\eta_{Z(\eta)}$ has been retained. Since the marks are
  independent and $\eta_{Z(\eta)}(\BX)=N$, we have
  \begin{displaymath}
    \BP(\eta_{Z(\eta)}\leq\eta^x\mid\eta)=x^N.
  \end{displaymath}
  Therefore,
  \begin{equation}
    \label{eq:firstprob}
    \BE x^N = \BP\left((\eta^x+(\eta-\eta^x))_{Z(\eta)}\leq\eta^x\right).
  \end{equation}
  Apply Lemma~\ref{lemma:thinning} with $\mu=\eta^x$ and $\psi=\eta-\eta^x$.
  The event on the right-hand side of \eqref{eq:firstprob} is exactly
  $\{(\eta-\eta^x)(Z_x)=0\}$.  Since
  $\eta^x$ and $\eta-\eta^x$ are independent Poisson
  processes
  and $Z_x$ is a function of $\eta^x$, we obtain
  \begin{displaymath}
    \BP((\eta-\eta^x)(Z_x)=0 \mid\eta^x)=\exp[-(1-x)\lambda(Z_x)].
  \end{displaymath}
  Taking expectations and using \eqref{eq:firstprob} yields
  \eqref{eq:BR-main}.

  Suppose that $x>0$ and $s\geq-1$. The spatial Markov property
  for $\eta^x$ gives
  \begin{displaymath}
    \BE\big[(s+1)^{\eta^x(Z_x^c)}\mid\eta^x_{Z_x}, \lambda_{Z_x}\big]
    =\exp\{sx\lambda(Z_x^c)\}.
  \end{displaymath}
  Here $Z_x=Z(\eta^x)$ is determined by $\eta^x_{Z_x}$ by
  self-replication. Consequently,
  \begin{align*}
    &e^{-sx\lambda(\BX)}
      \BE\big[(s+1)^{\eta^x(Z_x^c)}
      e^{(x-1+sx)\lambda(Z_x)}\big] \\
    &\qquad=e^{-sx\lambda(\BX)}
      \BE e^{sx\lambda(Z_x^c)+(x-1+sx)\lambda(Z_x)}
      =\BE e^{(x-1)\lambda(Z_x)}.
  \end{align*}
  Together with \eqref{eq:BR-main}, this proves \eqref{eq:BR-general}
  for $x>0$. If $x=0$, then $\BP(\eta^0=0)=1$, so
  that \eqref{eq:BR-general} simplifies to
  \begin{displaymath}
    \BE 0^{\eta(Z(\eta))}=\BE\exp\left\{-\lambda(Z(\eta^0))\right\},
  \end{displaymath}
  which is \eqref{eq:BR-main} with $x=0$.
\end{proof}

For a constant (deterministic) $Z$ the equation \eqref{eq:BR-general}
boils down to an elementary identity for the probability generating
function of Poisson random variables.  For $s=0$ we obtain
\eqref{eq:BR-main}, which is \cite[Theorem 4]{BeReitz15}.  Note also
that \eqref{eq:BR-general} for $x=1$ yields \eqref{e:transform} in the
case $t=-s$ and $B=\BX$.

\section{Further identities for Poisson hulls}
\label{s:furtherid}

It is difficult to extract from Theorem~\ref{t1} more
explicit information on the distribution of $(\eta_Z,\lambda_Z)$. In
this section we present some alternative relationships 
in the case where the stopping set is defined by a generator $\partial$. 
Recall that $\partial\eta=\eta_{[\eta]^c}$ as well as the definition
\eqref{eq:hull-partial} of $[\mu]$ and that $\mu\mapsto [\mu]^c$ is a
stopping set, see Lemma~\ref{l:hullstopping}.  In accordance with
Section~\ref{sec:hull-generator}, we define random measures
$\lambda_{[\eta]}:=\lambda(\cdot\cap [\eta])$ and
$\lambda_{[\eta]^c}:=\lambda(\cdot\cap [\eta]^c)$, as well as the
point process $\eta_{[\eta]}:=\eta(\cdot\cap [\eta])$.
Since $Z(\mu):=[\mu]^c$ is a stopping set and
$\eta(Z(\eta))=\partial\eta(\BX)$, 
we obtain from Corollary~\ref{c:stoppingMarkov}
\begin{displaymath}
  \BP\big(\eta_{[\eta]} \in\cdot\mid \partial\eta\big)
  =\Pi_{\lambda_{[\eta]}}(\cdot),
  \quad \BP\text{-a.s.}
\end{displaymath}

We next show that Theorem~5 in \cite{BeReitz15} remains valid for
general hulls and yields a formula for the factorial moments of
$\partial\eta(B)$. The result is the Poisson counterpart of Theorem~3
from the recent paper \cite{Buchta23}.

\begin{theorem} \label{t:2.7}
  Let $B\in\cX_0$ and $k\in\N$. Then
  \begin{multline}
    \label{e:2.36}
    \BE \partial\eta(B)^{(k)}
    =\BE\lambda_{[\eta]^c}(B)^k\\ 
    +\sum_{j=1}^{k-1}\binom{k}{j}(-1)^j \,
    \BE \int_{B^{k-j}}\I\big\{x_1\notin[\eta],\ldots,x_{k-j}\notin[\eta]\big\}
    \lambda_B\big([\eta+\delta_{\bx}]\setminus[\eta]\big)^j \,\lambda^{k-j}(d\bx).
  \end{multline}
\end{theorem}
\begin{proof}
  We follow \cite{BeReitz15}.  By the multivariate Mecke equation,
  \begin{align*}
    \BE \partial\eta(B)^{(k)}
    &=\BE \int_{B^k}
    \I\{x_1\notin[\eta],\ldots,x_k\notin[\eta]\}
      \,\eta^{(k)}(d\bx) \\
    &= \BE \int_{B^k}
      \I\{x_1\notin[\eta+\delta_{\bx}],\ldots,
      x_k\notin[\eta+\delta_{\bx}]\}
      \,\lambda^k(d\bx).
  \end{align*}
  Since $[\eta]\subset[\eta+\delta_{\bx}]$, the last indicator equals
  \begin{displaymath}
    \I\{x_1\notin[\eta],\ldots,x_k\notin[\eta]\}
    \prod_{i=1}^k
    \I\{x_i\notin[\eta+\delta_{\bx}]\setminus[\eta]\}.
  \end{displaymath}
  Using inclusion--exclusion for the product gives
  \begin{displaymath}
    \BE \partial\eta(B)^{(k)}
    =\BE \int_{B^k}\I\big\{x_1\notin[\eta],\ldots,x_{k}\notin[\eta]\big\}
    \sum_{I\subset[k]}(-1)^{|I|}\I\big\{\text{$x_j\in [\eta+\delta_{\bx}]$
    for $j\in I$}\big\} \,\lambda^{k}(d\bx).
  \end{displaymath}
  Here all the involved moments are finite because
  $\partial\eta(B)\leq\eta(B)$ and $\lambda(B)<\infty$ and
  $\partial\eta(B)\leq\eta(B)<\infty$ $\BP$-a.s. By Lemma~1 in the
  online supplement to \cite{LM23} (or by insertion tolerance of $Z$),
  for each $j\in[k]$ and each $\bx=(x_1,\ldots,x_k)\in\BX^k$, we have
  that $x_1,\ldots,x_j\in [\eta+\delta_{\bx}]$ if and only if
  $x_1,\ldots,x_j\in [\eta+\delta_{x_{j+1}}+\cdots+\delta_{x_k}]$.  In
  particular, the term with $I=[k]$ vanishes.  For $I=[k]$, the cited
  lemma gives $x_1,\ldots,x_k\in[\eta+\delta_{\bx}]$ if and only if
  $x_1,\ldots,x_k\in[\eta]$. This is incompatible with the factor
  $\I\{x_1\notin[\eta],\ldots,x_k\notin[\eta]\}$, and hence the
  corresponding term is zero.

  After relabelling the variables, a subset $I$ of cardinality $j$
  contributes the same amount as the subset $\{k-j+1,\ldots,k\}$.
  Integrating over these $j$ variables yields the factor
  \begin{displaymath}
    \lambda_B\big([\eta+\delta_{x_1}+\cdots+\delta_{x_{k-j}}]
    \setminus[\eta]\big)^j.
  \end{displaymath}
  Therefore, and by a symmetry argument,
  \begin{displaymath}
    \BE \partial\eta(B)^{(k)}
    =\BE \sum_{j=0}^{k-1}\binom{k}{j}(-1)^j
    \int_{B^{k-j}}\I\big\{x_1\notin[\eta],\ldots,x_{k-j}\notin[\eta]\big\}
    \lambda_B\big([\eta+\delta_{\bx}]\setminus[\eta]\big)^j\,\lambda^{k-j}(d\bx),
  \end{displaymath}
  and the result follows.
\end{proof}

In the special case $k=2$ equation \eqref{e:2.36} simplifies as follows. 
\begin{corollary}
  For $B\in\cX_0$ we have
  \begin{equation}
    \label{e:2.38}
    \BE \partial\eta(B)^{(2)}
    =\BE \lambda_{[\eta]^c}(B)^2
    -2 \,\BE \int_B \I\{x\notin[\eta]\}
    \lambda_B\big([\eta+\delta_x]\setminus[\eta]\big)\,\lambda(dx).
  \end{equation}
\end{corollary}

\begin{corollary}
  For $B\in\cX_0$ we have
  \begin{equation}
    \label{e:2.41}
    \Var \partial\eta(B)
    =\Var \lambda_{[\eta]^c}(B)+\BE \partial\eta(B) -2 \,\BE \int_B\I\{x\notin[\eta]\}
    \lambda_B\big([\eta+\delta_x]\setminus[\eta]\big)\,\lambda(dx).
  \end{equation}
\end{corollary}
\begin{proof}
  Since
  $\Var \partial\eta(B)=\BE \partial\eta(B)^{(2)}+\BE
  \partial\eta(B)-(\BE \partial\eta(B))^2$ and
  $\BE \partial\eta(B)=\BE \lambda_{[\eta]^c}(B)$ (see
  \eqref{e:Efrongeneral}), the result follows from \eqref{e:2.38}.
\end{proof}

Next we present some formulas for mixed factorial moments of
$\eta_{[\eta]}(B)$ and $\partial\eta(B)$.

\begin{theorem} \label{t:2.74}
Let $B\in\cX_0$, $k\in\N$ and $l\in\N_0$. Then
  \begin{displaymath}
    \BE \eta_{[\eta]}(B)^{(l)}\partial\eta(B)^{(k)}
    =\sum_{j=0}^{k-1}\binom{k}{j}(-1)^j
    \BE\int_{B^{k-j}} \lambda_B([\eta+\delta_{\bx}])^l
    \lambda_B\big([\eta+\delta_{\bx}]\setminus[\eta]\big)^j
    \,(\lambda_{[\eta]^c})^{k-j}(d\bx).
  \end{displaymath}
\end{theorem}
\begin{proof}
  Applying the spatial Markov property and then the multivariate Mecke
  equation yields
  \begin{align*}
    \BE \eta_{[\eta]}(B)^{(l)}\partial\eta(B)^{(k)}
    &=\BE \lambda_{[\eta]}(B)^l \int_{B^k}
      \I\big\{x_1\notin[\eta],\ldots,x_{k}\notin[\eta]\big\}\,\eta^{(k)}(d\bx)\\
    &=\BE \int_{B^k}\lambda\big(B\cap[\eta+\delta_{\bx}]\big)^l
      \I\big\{x_1\notin[\eta+\delta_{\bx}],\ldots,x_{k}
      \notin[\eta+\delta_{\bx}]\big\}\,\lambda^{k}(d\bx).
  \end{align*}
  As in the proof of Theorem~\ref{t:2.7} we can use here the
  inclusion--exclusion principle to obtain that the above equals 
  \begin{displaymath}
    \sum_{j=0}^{k-1}\binom{k}{j}(-1)^j
    \BE\int_{B^{k-j}}\I\big\{x_1\notin[\eta],\ldots,x_{k-j}\notin[\eta]\big\}
    \lambda_B\big([\eta+\delta_{\bx}]\big)^l
    \lambda_B\big([\eta+\delta_{\bx}]\setminus[\eta]\big)^j\,\lambda^{k-j}(d\bx).
  \end{displaymath}
  This concludes the proof.
\end{proof}

\section{Hulls with the prime property}
\label{sec:hulls-with-prime}

In this section we consider a generator $\partial$ 
satisfying the \emph{prime property}, that is,
\begin{equation}
  \label{e:prime}
  \I\big\{\partial(\mu+\delta_x)\ne\partial\mu\big\}
  =\prod_{y\in\mu}\I\big\{\partial(\delta_y+\delta_x)\ne \partial\delta_y\big\},
  \quad\mu\in\bN,\,x\in\BX.
\end{equation}
In terms of the stopping set $Z(\mu):=[\mu]^c$ 
(see \eqref{eq:hull-partial}), this means
\begin{displaymath}
  Z(\mu)=\bigcap_{y\in \mu} Z(\delta_y).
\end{displaymath}
Note that the prime property at $\mu=0$ yields
$\partial\delta_y=\delta_y$ for every $y$.  Examples can be found in
\cite{LM23} and are also mentioned later on.  Let us abbreviate
\begin{displaymath}
  H_x(\mu):=\I\big\{\partial(\mu+\delta_x)\ne\partial\mu\big\}
\end{displaymath}
and $\bH_x(\mu):=1-H_x(\mu)$.  If $\mu=\delta_y$, we write $H_x(y)$
and $\bH_x(y)$ for $H_x(\mu)$ and $\bH_x(\mu)$, respectively.  The
prime property is equivalent to
\begin{displaymath}
  \bH_x(\mu)=\max_{y\in\mu}\bH_x(y),\quad\mu\in\bN,\,x\in\BX,
\end{displaymath}
with the convention that the maximum over the empty set is $0$.

The moment generating function of 
$\lambda(B\cap[\eta]^c)$ can be expressed
as an infinite series of integrals involving only the
functions $H_x(y)$.  To shorten the statement, denote
\begin{displaymath}
  H_{\bx}(w):=H_{x_1}(w)\cdots H_{x_i}(w),
  \quad \bx=(x_1,\ldots,x_i)\in\BX^i,\,w\in\BX,\, i\in\N.
\end{displaymath}
For $i=0$ we set $H_{\bx}(w):=1$.

\begin{theorem}
  \label{t:mgfprime}
  Let $B\in\cX$ satisfy $\lambda(B)<\infty$ and let $t\in\R$.  Then
  \begin{displaymath}
    \BE e^{t\lambda_{[\eta]^c}(B)}
    =\sum^\infty_{i=0}\frac{ t^i}{i!}\int_{B^i} 
    \exp\Big[\int \big(H_{\bx}(w)-1 \big)\,\lambda(dw)\Big]\,\lambda^i(d\bx).
  \end{displaymath}
\end{theorem}
\begin{proof}
  Let $i\in\N$. We have
  \begin{displaymath}
    \BE \lambda_{[\eta]^c}(B)^i=
    \BE \int \I \big\{x_1,\ldots,x_i\in B\cap [\eta]^c\big\}\,\lambda^i(d\bx)
    =\BE \int_{B^i} H_{x_1}(\eta)\cdots H_{x_i}(\eta)\,\lambda^i(d\bx).
  \end{displaymath}
  At this stage we are using the prime property, giving
  \begin{displaymath}
    \BE  H_{x_1}(\eta)\cdots H_{x_i}(\eta)
    =\BE \prod_{w\in\eta}H_{x_1}(w)\cdots H_{x_i}(w). 
  \end{displaymath}
  From \cite[Exercise~3.6]{LastPenrose17} we obtain
  \begin{displaymath}
    \BE  H_{x_1}(\eta)\cdots H_{x_i}(\eta)
    =\exp\Big[\int \big(H_{\bx}(w)-1\big)\,\lambda(dw)\Big].
  \end{displaymath}
  Since by dominated convergence,
  \begin{displaymath}
    \BE e^{t\lambda_{[\eta]^c}(B)}
    =\sum^\infty_{i=0}\frac{ t^i}{i!}\BE \lambda_{[\eta]^c}(B)^i,
  \end{displaymath} 
  the assertion follows.
\end{proof}

Since $\lambda_{[\eta]}(B)=\lambda(B)-\lambda_{[\eta]^c}(B)$, we obtain
from Theorem~\ref{t:mgfprime}
\begin{displaymath}
  \BE e^{t\lambda_{[\eta]}(B)}=e^{t\lambda(B)}
  \sum^\infty_{i=0}\frac{(-1)^i t^i}{i!}\int_{B^i} 
  \exp\Big[\int \big(H_{\bx}(w)-1 \big)\,\lambda(dw)\Big]
  \,\lambda^i(d\bx),\quad t\in\R.
\end{displaymath}
This determines the distribution of $\lambda_{[\eta]}(B)$.  Moreover,
Theorem~\ref{t:jointtransfrom} yields a formula for the joint
transform of $(\eta_{[\eta]}(B),\lambda_{[\eta]}(B))$ on the left-hand
side of \eqref{e:transform}.

The proof of Theorem~\ref{t:mgfprime} yields the following corollary.

\begin{corollary}
  Let $B\in\cX_0$. Then
  \begin{align}
    \label{e:3.6}
    \BE \lambda_{[\eta]^c}(B)
    &=\int_B\exp\Big[-\int \bH_x(w)\,\lambda(dw)\Big]\,\lambda(dx),\\
    \label{e:3.7}
    \BE \lambda_{[\eta]^c}(B)^2 
    &=\int_{B^2} \exp\Big[\int \big(H_{x}(w)H_{y}(w)-1 \big)
      \,\lambda(dw)\Big]\,\lambda^2(d(x,y)).
  \end{align}
\end{corollary}

As a result we obtain explicit integral formulas for the variances of
$\lambda_{[\eta]^c}(B)$ and $\partial\eta(B)$.

\begin{corollary}
  \label{c:variances}
  For all $B\in\cX_0$,
  \begin{align}\notag
    \Var\lambda_{[\eta]^c}(B)
    &=\int_{B^2} \exp\Big[\int \big(H_{x}(w)H_{y}(w)-1\big)
      \,\lambda(dw)\Big]\,\lambda^2(d(x,y))\\
    \label{e:var1}
    &\qquad\qquad -\int_{B^2} \exp\Big[\int -
      \big (\bH_{x}(w)+\bH_{y}(w)\big) \,\lambda(dw)\Big]\,\lambda^2(d(x,y)),\\ 
    \notag
    \Var \partial\eta(B)
    &=\Var \lambda_{[\eta]^c}(B)
      +\int_B \exp\Big[-\int \bH_{x}(w)\,\lambda(dw)\Big]\,\lambda(dx)\\
    \label{e:var2}
    &\qquad -2\int_{B^2} \bH_y(x)\exp\Big[\int \big(H_{x}(w)H_y(w)-1\big)
      \,\lambda(dw)\Big]\,\lambda^2(d(x,y)).
  \end{align}
\end{corollary}
\begin{proof}
  Formula \eqref{e:var1} follows from \eqref{e:3.6} and \eqref{e:3.7}.
  To derive the second formula, we use \eqref{e:2.41}. To do so, we
  note that
  \begin{align*}
    \BE
    &\int_B\I\{x\notin[\eta]\}
      \lambda_B \big ([\eta+\delta_x]\setminus[\eta]\big)\,\lambda(dx)
    =\BE\int_{B^2} H_x(\eta)\bH_y(\eta+\delta_x)H_y(\eta)\,\lambda^2(d(x,y))\\
    &=\int_{B^2} \bH_y(x)\BE H_x(\eta) H_y(\eta)\,\lambda^2(d(x,y))\\
    &=\int_{B^2} \bH_y(x)\exp\Big[\int
      \big(H_x(w)H_{y}(w)-1\big)\,\lambda(dw)\Big]
      \,\lambda^2(d(x,y)).
  \end{align*}
  Furthermore, we obtain from \eqref{e:Efrongeneral} and \eqref{e:3.6}
  that
  \begin{displaymath}
    \BE \partial\eta(B)=\BE \lambda_{[\eta]^c}(B)
    =\BE \int_B H_x(\eta)\,\lambda(dx)
    =\int_B \exp\Big[-\int \bH_{x}(w)\,\lambda(dw)\Big]\,\lambda(dx).
  \end{displaymath}
  Hence, \eqref{e:var2} follows from \eqref{e:2.41} and  \eqref{e:3.7}.
\end{proof}

Next, we add a formula for the covariance between 
$\partial\eta(B)$ and $\lambda_{[\eta]^c}(B)$.

\begin{proposition}
  \label{p:cov}
  For all $B\in\cX_0$,
  \begin{align}
    \notag
    \Cov\big(\partial\eta(B),\lambda_{[\eta]^c}(B)\big)
    &=\int_{B^2} H_y(x)\exp\Big[\int \big(H_{x}(w)H_y(w)-1\big)
      \,\lambda(dw)\Big]\,\lambda^2(d(x,y))\\
    \label{e:cov}
    &\quad -\int_{B^2} \exp\Big[\int - \big (\bH_{x}(w)+\bH_{y}(w)\big)
      \,\lambda(dw)\Big]\,\lambda^2(d(x,y)).
  \end{align}
\end{proposition}
\begin{proof}
  We have
  \begin{displaymath}
    \BE\partial\eta(B)\lambda_{[\eta]^c}(B)
    =\BE  \lambda_{[\eta]^c}(B)\int _B H_x(\eta-\delta_x)\,\eta(dx).
  \end{displaymath}
  Here we have used the generator identity
  \begin{displaymath}
    \partial\mu(B)=\int_B H_x(\mu-\delta_x)\,\mu(dx),\quad \mu\in\bN;
  \end{displaymath}
  see \cite[Lemma~2.4 and eq.~(2.7)]{LM23}.
  By the Mecke equation we obtain
  \begin{displaymath}
    \BE\partial\eta(B)\lambda_{[\eta]^c}(B)
    =\BE\int_B  \lambda_{[\eta+\delta_x]^c}(B) H_x(\eta)\,\lambda(dx)
    =\BE\int_{B^2} H_y(\eta+\delta_x) H_x(\eta)\,\lambda^2(d(x,y)).
  \end{displaymath}
  Since $H_y(\eta+\delta_x)=H_y(x)H_y(\eta)$, this implies
  \begin{displaymath}
    \BE\partial\eta(B)\lambda_{[\eta]^c}(B)
    =\int_{B^2} H_y(x)\exp\Big[\int \big(H_{x}(w)H_y(w)-1\big)
    \,\lambda(dw)\Big]\,\lambda^2(d(x,y)).
  \end{displaymath}
  Finally, we can use $\BE \partial\eta(B)=\BE \lambda_{[\eta]^c}(B)$
  and \eqref{e:3.6} to conclude the proof.
\end{proof}

\begin{remark}
  Let $B\in\cX_0$. In the context of general hulls the random variable 
  \begin{displaymath}
    Y_B:=\lambda_{[\eta]}(B)+\partial\eta(B)
    =\lambda(B)+\partial\eta(B)-\lambda_{[\eta]^c}(B)
  \end{displaymath}
  was proposed in \cite{LM23} as an unbiased estimator of
  $\lambda(B)$. We have
  \begin{displaymath}
    \Var Y_B= \Var \partial\eta(B)
    +\Var \lambda_{[\eta]^c}(B)
    -2\Cov\big(\partial\eta(B),\lambda_{[\eta]^c}(B)\big).
  \end{displaymath}
  Inserting here the formulas from Corollary~\ref{c:variances} and
  \eqref{e:cov}, we arrive at the simple formula
  \begin{displaymath}
    \Var Y_B=\int_B\exp\Big[-\int \bH_x(w)\,\lambda(dw)\Big]\,\lambda(dx),
  \end{displaymath}
  in accordance with \cite[Theorem 5.1]{LM23}.
\end{remark}

\section{Convex hulls}
\label{section:convexhulls}

In this section we assume that $\BX=\R^d$ for some $d\in\N$.  Again we
take an increasing sequence $B_n\in\cB(\R^d)$, $n\in\N$, with union
$\R^d$ and define $\cX_0$ and $\bN\equiv\bN(\R^d)$ as in the
general case.  The \emph{convex hull} $\conv(A)$ of a set
$A\subset\R^d$ is the smallest convex set containing $A$.  If
$A=\emptyset$ then $\conv(A):=\emptyset$. For a counting measure $\mu$, we
denote by $\supp\mu$ the set of all $x$ such that $\mu(\{x\})\geq 1$
and write $x\in\mu$ if $x\in\supp\mu$. Given $\mu\in\bN$, we write
$\conv(\mu)$ for the convex hull of the support of $\mu$.

A \emph{vertex} of $\conv(\mu)$ is a point $x\in\mu$ such that
$x\notin \conv(\supp(\mu)\setminus\{x\})$. This
is the case if and only if $x$ is an \emph{extreme point} of
$\conv(\mu)$, that is, $x$ cannot be written as a nontrivial convex
combination of two different points from $\conv(\mu)$. A vertex has to
be a point from $\supp(\mu)$. We denote by $\partial\mu$ the
restriction of $\mu$ to the set of vertices of $\conv(\mu)$ and write
\begin{equation}
  \label{eq:mu-conv}
  [\mu]:=\conv(\mu)\setminus \supp(\partial \mu).
\end{equation}
Note that $\mu_{[\mu]^c}=\partial \mu$.

We shall show that $Z(\mu):=[\mu]^c$ satisfies the assumptions of
Theorem~\ref{t:spatMarkov2} and so is a Markov set. To this end we need
the following lemma, which is of independent interest. For
$K\in\cB(\R^d)$ we define $Z_K\colon \bN\to \cB(\R^d)$ by
\begin{align}
Z_K(\mu):=(K\setminus \conv(\mu_K))\cup \supp(\partial\mu_K)
  = K\cap[\mu_K]^c,\quad \mu\in\bN.
\end{align}

\begin{lemma}\label{l:23.18}
  Let $K\in\cB(\R^d)$. Then $Z_K$ is graph measurable. If
$K\in\cX_0$, then $Z_K$ is a stopping set.
\end{lemma}
\begin{proof}
  The restriction map $\mu\mapsto\mu_K$ is measurable, so it suffices
  to establish graph measurability for the convex hull and its
  vertices.  By \cite[Proposition~6.3]{LastPenrose17}, using a
  localising sequence, the atoms of $\mu$ can be enumerated measurably
  as $x_1(\mu),x_2(\mu),\ldots$.  For $k\in\N$ and $x\in\R^d$ set
  \begin{equation*}
    d_k(x,\mu):=\inf\bigg\{
    \bigg\|x-\sum_{j=1}^{k} a_jx_j(\mu)\bigg\|:
    a_j\in\mathbb Q_+,\ \sum_{j=1}^k a_j=1 \bigg\},
  \end{equation*}
  whenever $\mu(\R^d)\ge k$, and set $d_k(x,\mu):=\infty$ otherwise.
  Since this is an infimum over a countable set, $d_k(\cdot,\cdot)$ is
  measurable.  By Carath\'eodory's theorem, $x\in\conv(\mu)$ if and
  only if $d_k(x,\mu)=0$ for some $k\in\N$, proving the measurability
 of the convex hull map.

  Likewise, $x\in\supp\mu$ is a measurable condition. Indeed the above
  argument can easily be modified to yield that
  $x\in\conv(\supp\mu\setminus\{x\})$ is jointly measurable.  Hence
  $Z_K$ is graph-measurable.

  Now let $K\in\cX_0$, so that $\mu_K$ is finite. 
Then, additionally to $\mu_{Z_K(\mu)}=\partial\mu_K$, we also have
    $\conv(\supp\partial\mu_K)=\conv(\mu_K)$.
  If $\psi(Z_K(\mu))=0$, every
  atom of $\psi_K$ lies in $[\mu_K]$. Therefore,
  \begin{displaymath}
    \conv(\partial\mu_K+\psi_K)
    =\conv(\partial\mu_K)=\conv(\mu_K).
  \end{displaymath}
  All vertices of this polytope, with their original multiplicities,
  are retained, and $\psi_K$ adds no vertex. Consequently
  $Z_K(\partial\mu_K+\psi_K)=Z_K(\mu)$. Points outside $K$ do not
  affect $Z_K$, so this is exactly \eqref{eq:st-set-Z}.
\end{proof}

For infinite counting measures, the set $Z$ does not always satisfy
the stopping set property, see Example~\ref{ex:chull-problem}.
However, it is always finitely insertion tolerant.

\begin{theorem}\label{thr:chull-markov}
  The mapping $Z(\mu):=[\mu]^c$, $\mu\in\bN$, is finitely insertion
  tolerant and so is a Markov set.
\end{theorem}
\begin{proof}
  Put $Z_n(\mu):=B_n\cap[\mu_{B_n}]^c$ for a localizing sequence
  $(B_n)$.  By Lemma~\ref{l:23.18}, each $Z_n$ is a stopping set.  In
  view of Lemma~\ref{lem:limit-insertion}, the result follows from
  Corollary~\ref{c:spatMarkov2}, once we have shown that
  \eqref{approxstopping} holds.  Take $x\in\R^d$ and $\mu\in\bN$ and
  recall that $[\mu]^c$ is the disjoint union of $\conv(\mu)^c$ and
  the vertices of $\conv(\mu)$.  If $x\notin\conv(\mu)$, then
  $x\notin\conv(\mu_{B_n})$ for each $n\in\N$. Assume that $x$ is a vertex
  of $\conv(\mu)$ and $x\in B_n$ for some $n\in\N$.  Since
  $\conv(\mu_{B_n})\subset \conv(\mu)$ it clearly follows that $x$ is
  a vertex (extreme point) of $\conv(\mu_{B_n})$. Hence, if
  $x\notin [\mu]$, then \eqref{approxstopping} holds.  
  Assume that  $x\in [\mu]$.   By Carath\'eodory’s theorem,
  $x$ is the convex combination of distinct
  points from $\supp\mu\setminus\{x\}$.  For all sufficiently large $n\in\N$ these
  points are in $\supp\mu_{B_n}$, so that $x\in \conv(\mu_{B_n})$ but
  is no vertex of $\conv(\mu_{B_n})$.  Hence \eqref{approxstopping}
  holds in this case.
\end{proof}

If $\conv(\eta)$ is almost surely closed and does not
contain any line, then we can show that $\lambda_{[\eta]}$
is almost surely determined by $\partial\eta$. Hence
$\lambda_{[\eta]}$ can be dropped from the condition on the left-hand
side of \eqref{eMarkov2}, and we obtain
\begin{align}
  \label{eMarkov4}
  \BP(\eta_{[\eta]}\in A\mid \partial \eta)
  =\Pi_{\lambda_{[\eta]}}(A),\quad \BP\text{-a.s.}
\end{align}
Since the result is of independent interest, we state it as a theorem.

\begin{theorem}
  \label{thr:lambda-on-Z}
  Assume that $\conv(\eta)$ is almost surely closed and does not
  contain any line. Then the random measure $\lambda_{[\eta]}$ is
  almost surely determined by $\partial \eta$.
\end{theorem}
\begin{proof}
  If $\lambda(\R^d)<\infty$, the process is finite almost surely and
  $\conv(\eta)=\conv(\partial\eta)$.  The proof of Lemma~\ref{l:23.18}
  shows that $\mu\mapsto \conv(\partial\mu)$ is measurable.

  Suppose that $\lambda(\R^d)=\infty$. Then $\conv(\eta)$ is nonempty
  almost surely. By Lemma~\ref{l:recession-deterministic}, its
  recession cone $C_\lambda$ is deterministic. By assumption,
  $\conv(\eta)$ is a.s.\ closed and contains no line. The representation
  theorem \cite[Theorem~18.5]{MR274683} for such closed convex sets 
  gives a.s.\ that $\conv(\eta)=C(\partial\eta)$, where
  $C(\mu):=\cl\big(\conv(\mu)+C_\lambda\big)$, and $\cl$ denotes the
  topological closure.

  To verify measurability, fix a countable dense subset of $C_\lambda$
  and enumerate the atoms of an arbitrary $\mu\in\bN$ measurably.
  Finite rational convex combinations of these atoms, added to the
  chosen points of $C_\lambda$, form a countable dense subset of
  $\cl(\conv(\mu)+C_\lambda)$ when $\mu\ne 0$.  The distance between
  $x\in\R^d$ and the closed set $C(\mu)$ is therefore an infimum of
  countably many measurable functions of $(x,\mu)$.  This shows that
  $\mu\mapsto C(\mu)$ is graph-measurable.  Removing the measurably
  enumerated atoms of $\mu$ from $C(\mu)$ shows that 
  $C'(\mu):=C(\mu)\setminus\supp\mu$ is graph-measurable.
  Outside a $\BP$-null set we
  then have $[\eta]=C'(\partial\eta)$.  For each $B\in\cB(\R^d)$,
  integration of the indicator of $[\eta]$ against $\lambda_B$ shows
  that $\lambda_{[\eta]}(B)$ is $\sigma(\partial\eta)$-measurable up
  to a null set, as asserted. 
\end{proof}  

In the convex hull setting, Theorem~\ref{t:jointtransfrom} corresponds to
the Beermann--Reitzner identity for the number of inner points,
whereas Equation~\eqref{eq:BR-main} corresponds to their Theorem~4 for the
number of vertices.

\begin{example}[Vertices and boundary points of convex hulls of finite
  sets]
  \label{ex:chull}
  Let $\mu$ be a finite counting measure on $\BX=\R^d$ with $B_n\equiv \R^d$. Then
  $\conv(\mu)$ is a compact set. Fix a $k\in\{0,1,\dots,d-1\}$, and
  let $[\mu]_k$ be the union of the relative interiors of all its
  faces of dimension at least $k+1$. 
  Then $Z(\mu)=[\mu]_k^c$ is a stopping set
  and the restriction of $\mu$ onto $Z(\mu)$ is a generator denoted by
  $\partial_k\mu$.

  The case $k=0$ appeared as Example~2.1 in \cite{LM23}. Then
  $\partial\mu$ consists of the points from $\supp\mu$ which are
  vertices of $\conv(\mu)$ (accounting for possible multiplicities)
  and $[\mu]_0$ is $\conv(\mu)$ with the vertices removed.  Then our
  identities extend those from \cite{BeReitz15} by including the case
  of non-diffuse intensity measures.  If $k=d-1$, then
  $\partial_{d-1}\mu$ is the set of all points from $\mu$ which lie on
  the boundary of $\conv(\mu)$.  If the intensity measure $\lambda$ of
  a finite Poisson process $\eta$ is absolutely continuous with
  respect to Lebesgue measure, then, almost surely, no point of $\eta$
  lies in the relative interior of any face of the convex hull of
  dimension between $1$ and $d-1$. In this case the point processes
  $\partial_k\eta$, $k=0,\dots,d-1$, coincide almost surely.  New
  effects arise for intensity measures singular with respect to
  Lebesgue measure, even when they are diffuse. For instance, consider
  the space $\R^2$, and let $\lambda$ be the one-dimensional Hausdorff
  measure on the boundary of $[0,1]^2$. Then $\lambda([\eta]_0)>0$ if
  and only if at least two points of $\eta$ belong to the same segment
  on the boundary of the square, and $\eta([\eta]_0)$ is the number of
  points in the inner parts of the segments spanned by the extreme
  Poisson points on each side.  In this case $\partial_1\eta$
  coincides with $\eta$.
\end{example}

\begin{example}
  \label{ex:pole-zero}
  Let $\lambda$ be the measure on $\R^d$ with $\lambda(\{0\})=0$ and
  Lebesgue density $x\mapsto \|x\|^{-\alpha-d}$ on
  $\R^d\setminus\{0\}$, where $\alpha>0$ is some given parameter. Then
  $\lambda(B(0,\eps))=\infty$ for each $\eps>0$, while the complement
  of each such ball has finite mass.  We may choose
  $B_n:=\{0\}\cup B(0,n^{-1})^c$, $n\in\N$, to fit this example to our
  general setting. For every cone $C$ with vertex $0$ and positive
  angular measure, $\lambda(C\cap B(0,\eps))=\infty$.  Hence, almost
  surely $\eta$ has points arbitrarily close to $0$ in every such
  cone. Choose $d+1$ small cones whose directions positively span
  $\R^d$. Then we find a set $F$ of $d+1$ Poisson points arbitrarily
  close to $0$ whose convex hull $P$ contains $0$ in its interior.
  Outside $P$ there are only finitely many Poisson points.  Hence
  $\conv(\eta)=\conv((\supp\eta\cap P^c)\cup P)$ is a convex polytope.
  This fact also appears as Corollary~4.2 in \cite{kab:mar:tem:19}.
  Therefore, $\partial\eta$ is a finite point process and, as a
  consequence, $\conv\eta$ is the convex hull of
  $\supp(\partial\eta)$.  It follows that $\lambda_{[\eta]}$ is a.s.\
  $\sigma(\partial\eta)$-measurable.
\end{example}

\begin{example}[Failure of the stopping set property]
  \label{ex:chull-problem}
  Letting $Z(\mu)$ be the complement to $\conv(\mu)$ together with the
  extremal points of the convex hull does not necessarily result in a
  stopping set if $\mu$ is not finite. This phenomenon was missed in
  \cite{LM23} when discussing convex hulls examples. Indeed, let $\mu$
  be the counting measure on $\BX=\R$ which attaches unit mass to each
  of the natural numbers. Then $Z(\mu)=(-\infty,1]$ and
  $\partial\mu=\mu_Z=\delta_1$. However, letting $\psi$ be any finite
  counting measure on $Z(\mu)^c=(1,\infty)$, say $\psi=\delta_x$ with
  $x>1$, yields that
  $Z(\mu_{Z(\mu)}+\delta_x)=(-\infty,1]\cup[x,\infty)$ which differs
  from $Z(\mu)$. Correspondingly, the vertex-selection map $\partial\mu=\delta_1$
  does not satisfy property (H3).  The observed effect is explained by
  the fact that the infinite measure on the complement of $Z$ is
  replaced by a finite one.
\end{example}

\begin{example}[Infinite total intensity]
  Let $\lambda$ be the Lebesgue measure restricted to the set
  $Q:=\{(x,y):y\geq x^2\}$ on the plane $\BX=\R^2$.  In this case,
  $\conv(\eta)$ is almost surely a closed convex set containing no
  line. Its recession cone is the deterministic cone
  $\{0\}\times[0,\infty)$.  Moreover, the quadratic growth of the
  lower boundary of $Q$ produces vertices of $\conv(\eta)$ arbitrarily
  far in both horizontal directions. By Theorem~\ref{thr:lambda-on-Z}
  the random measure $\lambda_{[\eta]}$, is a.s.\ determined by
  $\partial\eta$. 
\end{example}

\begin{example}[Locally infinite intensity]
  Consider $\lambda$ from Example~\ref{ex:pole-zero} with $d=2$,
  however, restricted onto the upper half-plane. Almost surely, the
  origin is an accumulation point of $\eta$, but $0\notin\supp\eta$.
  The convex hull $\conv(\eta)$ is therefore not closed and
  $\partial\eta$ is an infinite point process. Although the closedness
  assumption of Theorem~\ref{thr:lambda-on-Z} is not satisfied, the
  missing boundary point causes no difficulty here. Indeed, the closed
  hull can be reconstructed from the extremal points together with the
  deterministic point $0$. Consequently, $\lambda_{[\eta]}$ is
  measurable with respect to $\partial\eta$.
\end{example}

\section{Examples}
\label{section:ex}

Here we describe several further examples of stopping sets, which lead
to various Efron-type formulas. 

\begin{example}[Generalization of convex hulls]
Take $\BX=\R^d$ and $B_n\equiv\R^d$, so configurations and
intensity measures are finite. Fix $m\in\N$ and put $S(\mu):=\supp\mu$.
Let $P_{-m}(\mu)$ be the intersection of all closed half-spaces
whose complements contain at most $m$ distinct points of $S(\mu)$.
Equivalently,
\begin{displaymath}
  P_{-m}(\mu)=
  \bigcap_{T\subset S(\mu),\; |S(\mu)\setminus T|\leq m}\conv(T),
\end{displaymath}
where $\conv(\varnothing)=\varnothing$. The set $P_{-m}(\mu)$ is a
possibly empty polytope. Define
\begin{displaymath}
  Z(\mu):=\big(\inter P_{-m}(\mu)\big)^c,
  \qquad \partial_{-m}\mu:=\mu_{Z(\mu)},
\end{displaymath}
using the ambient interior in $\R^d$. Thus all exterior and boundary
sample points are retained, with their multiplicities. 
By Lemma~\ref{l:23.18} and its proof $P_{-m}$ is graph-measurable.
Moreover, it is monotone increasing in its support.

Given $\mu\in\bN$ we verify the stopping set property.  If
$\inter P_{-m}(\mu)=\varnothing$, then $Z(\mu)=\R^d$ and there is
nothing to prove. Otherwise write $P:=P_{-m}(\mu)$, $S:=S(\mu)$ and
$R=S\setminus\inter P$. For $u\ne0$, let $q_{m+1}(u)$ be the
$(m+1)$-st largest value among $\langle x,u\rangle$, $x\in S$,
counting distinct support points.  The half-space
$H_u:=\{x:\langle x,u\rangle\leq q_{m+1}(u)\}$ is admissible, that is,
it occurs among the half-spaces defining $P$. Hence $P\subset H_u$ so
that $h(P,u)\leq q_{m+1}(u)$, where $h(P,\cdot)$ is the support
function of $P$.  Now take one of the at least $m+1$ points
$x\in S$ satisfying $\langle x,u\rangle\ge q_{m+1}(u)$.  Since
$h(P,u)\le q_{m+1}(u)$, such an $x$ cannot belong to $\inter P$.  Thus
these $m+1$ points all belong to $R=S\setminus\inter P$.  Consequently
any closed half-space that does not contain $P$ excludes at least
$m+1$ of these retained points.

Replace $\mu$ by $\nu=\mu_{Z(\mu)}+\psi_{Z(\mu)^c}$. Every
half-space admissible for $\mu$ contains $P$, and altering points
in $\inter P$ does not change its excluded support points. Hence
it remains admissible for $\nu$, giving $P_{-m}(\nu)\subset P$.
Conversely, the preceding directional argument shows that every
half-space admissible for $\nu$ contains $P$, so
$P\subset P_{-m}(\nu)$. Thus $Z(\nu)=Z(\mu)$. Duplicating a
support point leaves $Z$ unchanged. Lemma~\ref{lemma:generator-from-Z}
therefore proves that $\partial_{-m}$ is even a generator.
\end{example}

\begin{example}[Zero cell of a hyperplane tessellation]
Let $\BX$ be the Borel subset of the affine Grassmannian $A(d,d-1)$
consisting of hyperplanes that do not pass through the origin.
Localise by the sets of hyperplanes meeting a bounded Euclidean
ball. For $\mu\in\bN$, intersect the closed half-spaces bounded by
the hyperplanes in $\supp\mu$ and containing the origin, and denote
the resulting cell by $P(\mu)$. If $\mu=0$ then $P(\mu):=\R^d$.
Define
\begin{displaymath}
  Z(\mu):= \{H\in\BX:H\cap P(\mu)\ne\varnothing\}.
\end{displaymath}
It is not difficult to prove that $Z$ is a stopping set.

Let $\eta$ be a stationary Poisson hyperplane process with intensity
parameter $\gamma>0$ and even directional probability measure $Q$
whose support spans $\R^d$. Its zero cell $P(\eta)$ is bounded almost
surely. 
Under general position, $\eta(Z)=f_{d-1}(P(\eta))$ almost surely,
where $f_{d-1}(P(\eta))$ is the number of facets in
$P(\eta)$. Corollary~\ref{c2.5} and the normalization of
\cite{MR4807325} give
\begin{displaymath}
  \BE f_{d-1}(P(\eta))
  =\BE\lambda(Z(\eta))
  =2\gamma\BE\int_{\mathbb S^{d-1}}h(P(\eta),u)\,Q(du),
\end{displaymath}
which recovers \cite[Equation~(6.13)]{MR4807325}.
\end{example}

\begin{example}[Ball hulls]
  Let $\BX=\mathbb B^d$ be the closed unit Euclidean ball, and let
  $B_n\equiv \BX$. Thus, every configuration under consideration is
  finite and $\lambda(\BX)<\infty$.  For $S\subset\mathbb B^d$, let
  $\operatorname{bh}(S)$ (the ball hull of $S$) be the intersection of
  all translated unit balls containing $S$, with
  $\operatorname{bh}(\varnothing)=\varnothing$. Set
  \begin{displaymath}
    Z(\mu):=\mathbb B^d\setminus\inter\operatorname{bh}(\supp\mu),
  \end{displaymath}
  and $\partial\mu:=\mu_{Z(\mu)}$. The interior is taken in $\R^d$, so
  $\partial\mu$ retains the points of $\mu$ on the boundary, with
  their multiplicities.

  We now prove the stopping set property \eqref{eq:st-set-Z}.  Fix
  $\mu\in\bN$ and put $S:=\supp\mu$ and $K:=\operatorname{bh}(S)$. We
  first show that
  \begin{equation}
    \label{ballhullboundary}
    \operatorname{bh}(S\setminus\inter K)=K.
  \end{equation}
  For $A\subset\mathbb B^d$, let
  \begin{equation*}
    C(A):=\{c\in\R^d: A \subset B(c,1)\}=\bigcap_{y\in A}B(y,1)
  \end{equation*}
  be the set of feasible centers. This is a convex set and
  \begin{equation*}
    \operatorname{bh}(A)=\bigcap_{c\in C(A)}B(c,1).
  \end{equation*}
  Suppose that $x\in S\cap\inter K$. We claim that
  $x\in\operatorname{bh}(S\setminus\{x\})$.  Otherwise there is a
  $c_0\in C(S\setminus\{x\})$ such that $|x-c_0|>1$. On the other
  hand, $C(S)$ is nonempty, since $S\subset\mathbb B^d$, and we may
  choose $c_1\in C(S)$. By convexity, the segment joining $c_0$ and
  $c_1$ is contained in $C(S\setminus\{x\})$. Hence, by continuity,
  there is a point $c$ on this segment such that $\|x-c\|=1$.  It
  follows that $c\in C(S)$ and thus $K\subset B(c,1)$.  Since
  $x\in\partial B(c,1)$, this contradicts $x\in\inter K$.
  Consequently,
  $\operatorname{bh}(S\setminus\{x\})=\operatorname{bh}(S)=K$.  Since
  $S$ is finite, we may successively remove all points of
  $S\cap\inter K$. Since the ball hull remains unchanged at every
  step, this proves \eqref{ballhullboundary}.

  To verify \eqref{eq:st-set-Z} let $\psi$ be a finite counting
  measure satisfying $\psi(Z(\mu))=0$. By \eqref{ballhullboundary},
  $\operatorname{bh}(\supp\partial\mu)=K$.  Moreover,
  $\supp\psi\subset\inter K$. Every unit ball containing
  $\supp\partial\mu$ contains its ball hull $K$ and hence also
  $\supp\psi$. It follows that
  $\operatorname{bh}\big(\supp(\partial\mu+\psi)\big)=K$.  Therefore,
  $Z(\partial\mu+\psi)=Z(\mu)$, which is the stopping set property
  \eqref{eq:st-set-Z}. 

In the binomial setting, the Efron formula for ball hulls appears in
dimension two in \cite[Equation~(5.10)]{fod:kev:vig14}, in general
dimension in \cite[Equation~(1.3)]{fod20}, and for generalised hulls
in \cite[Page~500]{fod:pap:vig20}. Higher order Efron--Buchta
  identities for ball hulls are proved in \cite{bako-szabo26}.  Our
identities give Poisson counterparts and higher distributional
relationships in the finite configuration setting specified above.
\end{example}

\begin{example}[Pareto optimal points]
Consider either finite configurations on $\BX=\R^d$, with
$\cX_0=\cB(\R^d)$, or configurations on $\BX=\R_+^d$ that are
finite on bounded Borel sets. Write $y\leq x$ for the coordinatewise
order and put
\begin{displaymath}
  D_x:=\{y\in\BX:y\leq x,\ y\ne x\},\qquad
  Z(\mu):=\{x\in\BX:\mu(D_x)=0\}.
\end{displaymath}
The map $Z$ is graph measurable and monotone decreasing. The
restriction $\partial\mu=\mu_{Z(\mu)}$ consists of the Pareto
minimal points, retaining all multiplicities.

In either setting, for every point $x\in\supp\mu$ there exists a
Pareto minimal point $y$ with $y\leq x$.  For $\BX=\R_+^d$, this
follows because the lower box $\times_{j=1}^d[0,x_j]$ contains only
finitely many atoms. Thus the dominated region is reconstructed from
the retained points:
\begin{displaymath}
  Z(\mu)^c=\bigcup_{y\in\supp\partial\mu}\{x\in\BX:y\leq x,\ x\ne y\}.
\end{displaymath}
Deleting non-minimal points and inserting points in this region
preserves exactly the minimal support points. Indeed, every
inserted point is dominated by a retained minimum and cannot
dominate another retained minimum. This proves the stopping
property. Duplicating a support point does not change $Z$, so
Lemma~\ref{lemma:generator-from-Z} gives a generator. Its prime
property follows from
\begin{displaymath}
  \I\{x\in Z(\mu)\}=\prod_{y\in\mu}H_x(y)
  \quad \text{for}\quad H_x(y):=\I\{y\notin D_x\}.
\end{displaymath}
Theorem~\ref{t:spatMarkov2} yields the Markov property in the form
\eqref{SMP}.  The preceding formulas therefore apply with their
respective local finiteness and moment hypotheses.  The definition of
$\bN(\BX)$ is essential: an arbitrary infinite configuration on $\R^d$
can have a decreasing chain with no minimal point.
\end{example}

\begin{example}[Voronoi flower of the typical Voronoi cell]
\label{ex:flower}
Take $\BX=\R^d$ with the localising ring of bounded Borel sets.
For $\mu\in\bN$, let $C^o(\mu)$ be the cell of $0$ in the
Voronoi mosaic generated by $\supp(\mu+\delta_0)$. 
Define its \emph{Voronoi flower} $Z^o(\mu)$,
as the union of the closed balls
$B_{\|x\|}(x)$ over all points $x$ of $C^o(\mu)$, see
\cite{zuy92}.
It can be shown that $Z^o$ is a stopping set.

Let $\eta$ be a stationary Poisson process of intensity $\gamma>0$.
Then $C^o=C^o(\eta)$ is the typical Voronoi cell. 
It is almost surely a (bounded) polytope.
If $N^o$ is the number of
$(d-1)$-dimensional facets of $C^o$, general position implies that
$\eta(Z^o)=N^o$ almost surely.
The flower balls have no Poisson points in their interiors, and
their Poisson boundary points are exactly the facet neighbours.
By \cite[Example~5(ii)]{MR1404304}, conditionally on $N^o=m$,
$\lambda_d(Z^o)$ is Gamma distributed with shape $m$ and rate
$\gamma$, equivalently scale parameter $1/\gamma$.
It follows that $\BE \lambda_d(Z^o)=\gamma^{-1}\BE N^o$.
Together with the Efron
type identity \eqref{e:Efrongeneral}, this yields that
\begin{align}
  \label{meanflower1}
  \BE \eta(Z^o) =\BE N^o=\gamma\BE\lambda_d(Z^o),
\end{align}
relating the number of Poisson points in the flower with the number of
facets of $C^o$. The first equality trivially follows from the fact
that $\BP(N^o=\eta(Z^o))=1$; see \cite{MR1404304}.  In dimension two,
$\BE N^o=6$; see \cite[Theorem~10.2.5]{SW08}.  In the case $d=3$ we
obtain from \cite[Theorem~10.2.5]{SW08} the explicit formula
\begin{displaymath}
  \BE N^o=2+48\pi^2/35.
\end{displaymath}
Each facet of $C^o$ corresponds to a point from $\eta_{Z^o}$, so that
$\eta_{Z^o}$ can be interpreted as the point process of facets of
$C^o$.  Corollary~\ref{c2.5} says that $\eta_{Z^o}$ has the same
intensity measure as $\gamma\lambda_d(Z^o\cap\cdot)$. Furthermore, the
signed random measure
$I_{Z^o}:=\eta_{Z^o}-\gamma\lambda_d(Z^o\cap\cdot)$ satisfies the
isometry property \eqref{e:isometry}; see Remark~\ref{r:signedsecond}.
These local properties of $Z^o$ cannot be derived from the conditional
Gamma distribution of $\lambda_d(Z^o)$.
\end{example}

\begin{example}[Voronoi flower of the Voronoi zero cell]\label{ex:flowerzero}
Use the same Euclidean localization as in Example~\ref{ex:flower}.
For $\mu\ne0$, let $S(\mu)$ be the nearest support point to $0$,
breaking ties lexicographically, and let $C(\mu)$ be its Voronoi
cell in the mosaic generated by $\supp\mu$. 
Let $Z(\mu)$ be the Voronoi flower of $S(\mu)$. 
If $\mu=0$ we set $C(\mu)=Z(\mu)=\R^d$.
Then $Z$ is a stopping set, as can be derived
from the stopping set property of $Z^o$ and translation
equivariance of the Voronoi cell.

For a stationary Poisson process $\eta$ with intensity $\gamma>0$,
the random set $C(\eta)$ is known as the
\emph{zero cell} of the Voronoi mosaic generated by $\eta$.
It is almost surely a polytope.
Writing $N$ for its number of facets, general position gives
$\eta(Z)=N+1$ almost surely: the extra point is the nucleus.
By \cite[Remark~5.3]{BaumLast09}, conditionally on $N=m$,
$\lambda_d(Z)$ is Gamma distributed with shape $m+1$ and rate
$\gamma$, equivalently scale $1/\gamma$. Hence
\begin{equation}\label{meanflower2}
  \BE\eta(Z)=1+\BE N=\gamma\BE\lambda_d(Z),
\end{equation}
which agrees with \eqref{e:Efrongeneral} and should be compared with
\eqref{meanflower1}. This follows more directly from
$\BP(N+1=\eta(Z))=1$; see e.g.\ \cite{BaumLast09}.
\end{example}

\begin{example}[Nearest neighbor to the origin]
Take $\BX=\R^d$ with the localising ring of bounded Borel sets,
and fix $k\in\N$. For $\mu(\R^d)\geq k$, define
\begin{displaymath}
  R_k(\mu):=\inf\{r\geq0:\mu(B_r(0))\geq k\},
  \qquad Z(\mu):=B_{R_k(\mu)}(0),
\end{displaymath}
counting points with multiplicity and using closed balls. If
$\mu(\R^d)<k$, set $Z(\mu)=\R^d$. Euclidean local finiteness
ensures that $R_k$ is attained whenever it is needed. 
It is easy to check that $Z$ is a stopping set.

When $\mu(\R^d)\geq k$, the possibility of ties yields that 
\begin{displaymath}
  \mu(\inter Z(\mu))<k\leq\mu(Z(\mu)),\qquad
  \mu(Z(\mu))=\mu(\inter Z(\mu))+\mu(\partial Z(\mu)).
\end{displaymath}
In particular, the number in the interior need not be $k-1$.
Suppose now that $\lambda(\R^d)=\infty$, 
and that $\lambda(\{x:\|x\|=r\})=0$ for every
$r\geq0$. Then radial ties have probability zero and $\eta(Z(\eta))=k$
almost surely. The function $\Lambda(r):=\lambda(B_r(0))$ is
finite and continuous, starts at zero, and increases to infinity.
Mapping each point $x$ to $\Lambda(\|x\|)$ gives a Poisson process
of unit intensity on $\R_+$. Intervals on which $\Lambda$ is
constant carry no radial intensity. Therefore,
\begin{displaymath}
  \lambda(Z(\eta))=\Lambda(R_k(\eta))
  \ \sim\ \operatorname{Gamma}(k,1),
  \qquad \BE\lambda(Z(\eta))=k.
\end{displaymath}
The expectation is also \eqref{e:Efrongeneral}. For the stationary
case, see \cite{MR1404304}.
\end{example}

\begin{example}[Embracing graph]
Take $\BX=\R^d$ with the localising ring of bounded Borel sets.
Enumerate the distinct support points $x_1(\mu),x_2(\mu),\ldots$
by increasing distance to $0$, breaking ties lexicographically.
Let $N(\mu)$ be the first index for which the positive hull of
$x_1(\mu),\ldots,x_{N(\mu)}(\mu)$ is $\R^d$. If such an index
exists, set $Z(\mu)=B_{\|x_{N(\mu)}(\mu)\|}(0)$; otherwise set
$Z(\mu)=\R^d$. All points at the terminal radius are contained in $Z(\mu)$.
The finite positive-spanning condition is measurable. Since changes
outside the revealed ball cannot affect the first spanning
radius, $Z$ is a stopping set.

For a stationary Poisson process of intensity one in dimension two,
ties are absent and $\eta(Z)=N(\eta)$ almost surely. The identities
$\BE N(\eta)=5$ and $\Var N(\eta)=4$ (at the moment we do not use this
variance formula) are given in \cite{chiu:mol03}.  Thus
\eqref{e:Efrongeneral} gives
\begin{displaymath}
  \BE\lambda(Z)=\pi\BE\|x_{N(\eta)}(\eta)\|^2=5.
\end{displaymath}
\end{example}

\begin{example}[Sums]\label{ex:sums}
  Let $\BX=\R_+\times\R$, with $\cX_0$ being the ring of all Borel
  sets contained in $[0,t]\times\R$ for some finite $t$. Thus every
  $\mu\in\bN$ is finite on each such time slab.  Enumerate the atoms,
  with multiplicities, in lexicographic order: first by nondecreasing
  first coordinate $x_i$, interpreted as time, and, among atoms having
  the same first coordinate, by nondecreasing second coordinate
  $y_i$. Set
  \begin{displaymath}
    S_i:=\sum_{j=1}^i y_j,\qquad T_i:=\sum_{j=1}^i x_j,
  \end{displaymath}
  and fix $a>0$. If there is an index $m$ with $S_m\geq a$, take the
  first such index and define $Z(\mu)=[0,T_m]\times\R$.  If there is
  no such index, including when $\mu=0$, set $Z(\mu)=\BX$.  This
  construction is graph measurable. Since $T_m\geq x_m$, the revealed
  slab contains every atom used to determine the first crossing, and
  all atoms determining the first crossing are retained by restriction
  to $Z(\mu)$.  Any atoms added to $Z(\mu)^c$ have time coordinate
  greater than $T_m$, hence occur after the first atoms. Thus
  neither the first crossing nor $T_m$ changes. This proves the
  stopping set property.

  The map $Z$ is not monotone decreasing. For example, with $a=1$ and
  $\mu=\delta_{(1,1)}$ the terminal time is $1$. Adding
  $\delta_{(1/2,-1)}$ prevents any crossing, so the stopping set
  becomes all of $\BX$.
\end{example}

\begin{example}[Functions attaining a pointwise maximum]
  Let $T$ be a nonempty compact metric space, let $\BX$ be the space
  $C(T)$ of real-valued continuous functions with the uniform metric
  and its usual Borel $\sigma$-algebra, and take $\cX_0=\cX$, so that
  configurations and intensity measures are finite. For $\mu\ne0$,
  write
  \begin{displaymath}
    g_\mu(t):=\max_{f\in\supp\mu}f(t),\qquad
    Z(\mu):=\{f\in C(T):\max_{t\in T}(f(t)-g_\mu(t))\geq0\}.
  \end{displaymath}
  For $\mu=0$ we set $Z(0):=\BX$. The envelope $g_\mu$ is
  continuous. The supremum over a fixed countable dense subset of $T$ equals
  the maximum over $T$, so the finite-atom enumeration proves
  graph measurability of $Z$.

  Define $\partial\mu=\mu_{Z(\mu)}$. It retains every sample function
  that attains the envelope somewhere, including ties and all
  multiplicities. At each $t$, at least one of the finitely many
  sample functions attains $g_\mu(t)$ and is retained.  Consequently
  $g_{\partial\mu}=g_\mu$. Deleting the remaining functions and
  inserting functions strictly below $g_\mu$ at every point of $T$
  preserves this envelope and $Z$. This is the stopping set
  property. The map $Z$ is monotone decreasing, and duplicating a
  sample function leaves it
  unchanged. Lemma~\ref{lemma:generator-from-Z} therefore yields a
  generator, with
  \begin{displaymath}
    [\mu]=\{f\in C(T): f(t)<g_\mu(t)\text{ for every }t\in T\},
    \qquad [0]=\varnothing.
  \end{displaymath}
  Our identities relate the number of functions attaining the
  envelope, the number strictly below it, and the intensity measure of
  this strict domination class.
\end{example}

\appendix
\refstepcounter{section}

\section*{Appendix \thesection: Relations between stopping set, hull
  and generator}
\label{sec:appendix}

A measurable map $\partial\colon\bN\to\bN$ is said to be a
\emph{generator} if it satisfies the following properties listed in
\cite{LM23}:
\begin{enumerate}
\item[(H1)] $\partial\mu\leq \mu$ (thinning);
\item[(H2)] for all $\mu\in\bN$ and $\partial\mu(\{x\})\geq1$, we have
  $\partial(\mu+\delta_x)=\partial\mu + \delta_x$ (additivity);
\item[(H3)] for all $\mu,\mu'\in\bN$ such that
  $\mu'\leq \mu-\partial\mu$, we have
  $\partial(\partial\mu+\mu')=\partial\mu$ (idempotency);
\item[(H4)] if $\mu,\mu'\in\bN$ satisfy $\mu'\leq\mu$ and
  $\partial\mu=\partial\mu'$, then
  $\partial(\mu+\psi)=\partial(\mu'+\psi)$ for all $\psi\in\bN$
  (consistency).
\end{enumerate}

The associated \emph{hull operator} is the map from $\bN$ to $\cX$ defined by
\eqref{eq:hull-partial}.
We first show that each generator produces a stopping set.

Since $\partial$ is measurable and $(x,\mu)\mapsto \mu+\delta_x$ is
measurable, the set
\begin{displaymath}
  \{(x,\mu):\partial(\mu+\delta_x)=\partial\mu\}
\end{displaymath}
is measurable. Thus, the hull map $\mu\mapsto[\mu]$ is graph
measurable, and so is $Z(\mu)=[\mu]^c$.

\begin{lemma}\label{l:hullstopping}
  Assume that $\partial\colon\bN\to\bN$ is a generator.  Define
  $Z\colon\bN\to\cX$ by
  \begin{displaymath}
    Z(\mu):=[\mu]^c,\quad \mu\in\bN.
  \end{displaymath}
  Then \eqref{eq:st-set-Z} holds for each $\mu\in\bN$ and each
  $\psi\in\bN$ and $Z(\mu)$ is monotone decreasing in $\mu$.
\end{lemma}
\begin{proof}
  By Lemma~2.7 from \cite{LM23}, $\nu_{Z(\nu)}=\kappa$ for
  $\nu,\kappa\in\bN$ is equivalent to $\nu_{Z(\kappa)}=\kappa$ and
  implies $[\nu]=[\kappa]$. Let $\kappa=\partial\mu$ and
  $\nu=\partial\mu+\psi_{[\mu]}$ for any $\psi\in\bN$. Then
  $Z(\kappa)=Z(\partial\mu)=Z(\mu)$ and
  \begin{displaymath}
    \nu_{Z(\kappa)}=\big(\partial\mu+\psi_{[\mu]}\big)_{Z(\mu)}=\partial\mu=\kappa.
  \end{displaymath}
  Hence, taking into account Lemma~2.5 from \cite{LM23} for the last
  equality,
  \begin{displaymath}
    [\nu]=\big[\partial\mu+\psi_{[\mu]}\big]=[\kappa]=[\partial\mu]=[\mu].
  \end{displaymath}
  Therefore, $Z(\partial\mu+\psi_{[\mu]})=Z(\mu)$, which is exactly the
  definition property \eqref{eq:st-set-Z} of a stopping set.
  Finally, note that $\mu\mapsto[\mu]$ is monotone increasing, since
  $\partial\mu=\partial(\mu+\delta_x)$ implies
  $\partial(\mu+\psi)=\partial(\mu+\psi+\delta_x)$ by (H4). 
\end{proof}

Conversely, let $Z\colon\bN\to\cX$ be a stopping set. Motivated by
Lemma~\ref{l:hullstopping}, define
\begin{equation}
  \label{eq:Zgen}
  \partial\mu:=\mu_{Z(\mu)},\quad \mu\in\bN.
\end{equation}
Note that a stopping set is not necessarily monotone decreasing,
meaning that \eqref{eq:Zgen} does not necessarily satisfy all
properties (H1)-(H4); see Example~\ref{ex:sums}. 

\begin{lemma}
  \label{lemma:generator-from-Z}
  Let $Z\colon\bN\to\cX$ be a stopping set such that
  \begin{equation}
    \label{eq:Zplusx}
    Z(\mu+\delta_x)=Z(\mu), \quad x\in\mu_{Z(\mu)}.
  \end{equation}
  Then the map given at \eqref{eq:Zgen} satisfies (H1)--(H3).
  If $Z$ is monotone decreasing, then (H4) also holds and the
  corresponding hull operator given by $[\mu]:=Z(\mu)^c$ satisfies
  \eqref{eq:hull-partial}.
\end{lemma}
\begin{proof}
  Condition (H1) is trivial.
  Condition (H2) follows from assumption \eqref{eq:Zplusx}
  imposed on the stopping set: if
  $x\in\partial\mu$, then
  \begin{displaymath}
    \big(\mu+\delta_x\big)_{Z(\mu+\delta_x)}
    =\mu_{Z(\mu)}+\delta_x. 
  \end{displaymath}  
  To check (H3) we take $\mu,\mu'\in\bN$ with
  $\mu'\leq \mu-\partial\mu$.  Then $\mu'\leq \mu_{Z(\mu)^c}$, and so
  \begin{displaymath}
    Z(\partial\mu+\mu')=Z\big(\mu_{Z(\mu)}+\mu'_{Z(\mu)^c}\big)=Z(\mu).
  \end{displaymath}
  Therefore,
  \begin{displaymath}
    \partial\big(\partial\mu+\mu'\big)
    =\big(\partial\mu+\mu'\big)_{Z(\partial\mu+\mu')}
    =\big(\partial\mu+\mu'\big)_{Z(\mu)}=\partial\mu,
  \end{displaymath}
  proving (H3).

  To prove (H4) we first record a consequence of the stopping
  property. If $\nu,\kappa\in\bN$ agree on $Z(\nu)$, then
  $\kappa=\nu_{Z(\nu)}+\kappa_{Z(\nu)^c}$, so \eqref{eq:st-set-Z}
  gives $Z(\kappa)=Z(\nu)$. Assume now that $\mu'\leq\mu$ and
  $\partial\mu'=\partial\mu$. Self-replication yields
  \begin{displaymath}
    Z(\mu')=Z(\partial\mu')=Z(\partial\mu)=Z(\mu),
  \end{displaymath}
  and hence $\mu_{Z(\mu)}=\mu'_{Z(\mu')}=\mu'_{Z(\mu)}$.
  Fix $\psi\in\bN$ and put $\nu=\mu+\psi$, $\kappa=\mu'+\psi$.
  By monotonicity, $Z(\nu)\subset Z(\mu)$. Since $\nu$ and $\kappa$
  agree on $Z(\mu)$, they agree on $Z(\nu)$, and the preceding
  observation implies $Z(\kappa)=Z(\nu)$. Therefore,
  \begin{displaymath}
    \partial(\mu+\psi)=\nu_{Z(\nu)}
    =\kappa_{Z(\kappa)}=\partial(\mu'+\psi),
  \end{displaymath}
  as required.

  It remains to prove \eqref{eq:hull-partial}. We show that
  \begin{displaymath}
    Z(\mu) = \{x\in\BX:\partial(\mu+\delta_x)\ne\partial\mu\}.
  \end{displaymath}
  First assume that $x\notin Z(\mu)$. Then, by the stopping-set
  property, $Z(\mu+\delta_x)=Z(\mu)$.  Consequently,
  \begin{displaymath}
    \partial(\mu+\delta_x)
    = (\mu+\delta_x)_{Z(\mu+\delta_x)}
    = (\mu+\delta_x)_{Z(\mu)}
    = \mu_{Z(\mu)}
    = \partial\mu.
  \end{displaymath}
  Conversely, assume that $x\in Z(\mu)$. By monotonicity,
  $Z(\mu+\delta_x)\subset Z(\mu)$.  We first claim that
  $x\in Z(\mu+\delta_x)$. If not, then $\mu$ and $\mu+\delta_x$ agree
  on $Z(\mu+\delta_x)$. By the observation used in the proof of (H4),
  this would imply $Z(\mu)=Z(\mu+\delta_x)$, 
  contradicting $x\in Z(\mu)$ and $x\notin Z(\mu+\delta_x)$. Hence
  $x\in Z(\mu+\delta_x)$. Therefore,
  \begin{displaymath}
    \partial(\mu+\delta_x)(\{x\})
    = \mu(\{x\})+1,
    \qquad
    \partial\mu(\{x\})
    = \mu(\{x\}),
  \end{displaymath}
  and so $\partial(\mu+\delta_x)\ne\partial\mu$.
\end{proof}

\refstepcounter{section}

\section*{Appendix \thesection: Insertion tolerance}
\label{sec:appendtolerance}

\begin{lemma}\label{lem:insertion-tolerance}
  Let $Z\colon\bN\to\cX$ be graph-measurable. Then the following
  holds.
  \begin{enumerate}[(ii)]
  \item[{\rm (i)}] If $Z$ is finitely insertion tolerant and
    $\mu(Z(\mu)^c)<\infty$ for some $\mu\in\bN$, then
    $Z(\mu_{Z(\mu)})=Z(\mu)$.
  \item[{\rm (ii)}] The mapping $Z$ is insertion tolerant if and only if
    it is a stopping set.
  \end{enumerate}
\end{lemma}
\begin{proof}
  To prove (i), put $\nu:=\mu_{Z(\mu)}$ and $\psi:=\mu_{Z(\mu)^c}$.
  Then $\psi$ is finite, $\mu=\nu+\psi$, and $\psi(Z(\mu))=0$.  By
  finite insertion tolerance, applied to $\nu$ and $\psi$, we have
  $\psi(Z(\nu))=0$ and $Z(\nu)=Z(\nu+\psi)=Z(\mu)$.  Since
  $\nu=\mu_{Z(\mu)}$, this proves (i).

  For (ii), suppose first that $Z$ is insertion tolerant. For
  $\mu\in\bN$, put again $\nu:=\mu_{Z(\mu)}$ and
  $\psi:=\mu_{Z(\mu)^c}$.  Then $\mu=\nu+\psi$ and
  $\psi(Z(\mu))=0$. Insertion tolerance gives
  $Z(\mu_{Z(\mu)})=Z(\nu)=Z(\mu)$.  Thus $Z$ is self-replicating.  Now
  let $\mu,\psi\in\bN$ satisfy $\psi(Z(\mu))=0$.  By self-replication,
  $Z(\mu_{Z(\mu)})=Z(\mu)$, and hence
  $\psi\bigl(Z(\mu_{Z(\mu)})\bigr)=0$.  Insertion tolerance, applied
  to $\mu_{Z(\mu)}$ and $\psi$, yields
  \[
    Z(\mu_{Z(\mu)}+\psi)
    =Z(\mu_{Z(\mu)})
    =Z(\mu).
  \]
  Hence $Z$ is a stopping set.

  Conversely, suppose that $Z$ is a stopping set and let
  $\mu,\psi\in\bN$. If $\psi(Z(\mu))=0$, then
  \[
    \mu+\psi
    =\mu_{Z(\mu)}+\bigl(\mu_{Z(\mu)^c}+\psi\bigr),
  \]
  where the second term vanishes on $Z(\mu)$. The stopping set
  property therefore gives $Z(\mu+\psi)=Z(\mu)$, hence
  $\psi(Z(\mu+\psi))=0$.  To show the reverse implication in
    \eqref{insertion-equivalence}, suppose that $\psi(Z(\mu+\psi))=0$
  and put $\nu:=\mu+\psi$. Then $\mu_{Z(\nu)}=\nu_{Z(\nu)}$.
  Moreover, $\mu=\nu_{Z(\nu)}+\mu_{Z(\nu)^c}$, where $\mu_{Z(\nu)^c}$
  vanishes on $Z(\nu)$. The stopping set property therefore gives
  $Z(\mu)=Z(\nu)=Z(\mu+\psi)$.  Consequently, $\psi(Z(\mu))=0$. Thus
  $Z$ is insertion tolerant.
\end{proof}

\begin{lemma}  \label{lemma:thinning}
  Let $Z$ be a stopping set. Then, for each $\mu,\psi\in\bN$, we have
  $(\mu+\psi)_{Z(\mu+\psi)}\leq \mu$ if and only if $\psi(Z(\mu))=0$.
\end{lemma}
\begin{proof} 
  Suppose that $\mu,\psi\in\bN$ satisfy
  $(\mu+\psi)_{Z(\mu+\psi)}\leq \mu$.  Then
  $\mu_{Z(\mu+\psi)}+\psi_{Z(\mu+\psi)}\le \mu_{Z(\mu+\psi)}$, hence
  $\psi(Z(\mu+\psi))=0$. If the latter holds, then
  $(\mu+\psi)_{Z(\mu+\psi)}\leq \mu$. The assertion follows from
  Lemma~\ref{lem:insertion-tolerance}. 
\end{proof}

\begin{lemma}\label{lem:limit-insertion}
  Let $Z_n\colon\bN\to\cX$, $n\in\N$, be finitely insertion tolerant
  and suppose that $Z\colon \bN\to\cX$ satisfies
  \eqref{approxstopping}.  Then $Z$ is finitely insertion tolerant.
\end{lemma}
\begin{proof}
  Fix $\mu\in\bN$ and a finite $\psi\in\bN$. Suppose first that
  $\psi(Z(\mu))=0$. Since $\psi$ has finite support,
  \eqref{approxstopping} implies that $\psi(Z_n(\mu))=0$ for all
  sufficiently large $n$. Finite insertion tolerance of $Z_n$ then
  gives $Z_n(\mu+\psi)=Z_n(\mu)$ for all sufficiently large
  $n$. Passing to the pointwise graph limit yields
  $Z(\mu+\psi)=Z(\mu)$.  In particular, $\psi(Z(\mu+\psi))=0$.

  Assume now that $\psi(Z(\mu+\psi))=0$. Applying
  the same argument to $\mu+\psi$ shows that $\psi(Z_n(\mu+\psi))=0$
  for all sufficiently large $n$. By finite insertion tolerance of
  $Z_n$, we have $\psi(Z_n(\mu))=0$ and $Z_n(\mu+\psi)=Z_n(\mu)$ for
  all sufficiently large $n$. Passing to the limit gives
  $\psi(Z(\mu))=0$ and $Z(\mu+\psi)=Z(\mu)$, finishing the proof.
\end{proof}

\refstepcounter{section}
\section*{Appendix \thesection: Convex hulls of infinite Poisson processes}
\label{sec:case-infin-intens}

Throughout this appendix assume $\lambda(\R^d)=\infty$,
with the localization of Section~\ref{section:convexhulls}. Then
$P:=\conv(\eta)$ is nonempty almost surely. Write $\overline{P}=\cl P$
for the closure of $P$ and define its \emph{recession cone} by
\begin{displaymath}
  C_\lambda:=\rec(\overline{P})
  =\{v\in\R^d:y+tv\in\overline{P}\text{ for all }y\in\overline{P}, t\geq0\}.
\end{displaymath}
For a nonempty closed convex set $K$, its \emph{barrier cone} is
$D(K):=\{u:h(K,u)<\infty\}$, where
$h(K,u):=\sup_{x\in K}\langle x,u\rangle$ denotes the support
function. Recall that the \emph{polar} of a cone $C$ is
\begin{displaymath}
  C^o:=\{u:\langle x,u\rangle\leq0\text{ for every }x\in C\}.
\end{displaymath}
Convex duality gives $\rec(K)=D(K)^o$ and $\cl D(K)=(\rec K)^o$ for
each nonempty closed convex set $K$, see
\cite[Corollary~14.2.1]{MR274683}.  In particular, the recession cone
is closed, while the barrier cone need not be closed.

\begin{lemma}\label{l:recession-deterministic}
  The recession cone $C_\lambda$ of $\overline{P}$ is deterministic.
\end{lemma}
\begin{proof}
  For $R>0$, let $B_R(0)$ be the closed Euclidean ball and put
  \begin{displaymath}
    K_R:=\cl\conv(\supp\eta_{B_R(0)^c}\cup\{0\}).
  \end{displaymath}
  Removing points in a bounded ball and adding the origin does not
  change whether a support function is finite in any direction.  This
  remains true even if the ball contains infinitely many points under
  the chosen localization. Hence $D(\overline{P})=D(K_R)$ and
  $\rec(\overline{P})=\rec(K_R)$.

  This equality also gives the required random closed set
  measurability. The closed convex hull $K_R$ is measurable: its
  distance function is obtained using countably many finite rational
  convex combinations of measurably enumerated atoms and the
  origin. Since $0\in K_R$ and $K_R$ is closed and convex,
  \begin{displaymath}
    \rec(K_R)=\bigcap_{n\geq1}n^{-1}K_R.
  \end{displaymath}
  These are decreasing nonempty closed sets, and their distance functions
  increase to the distance function of the intersection. Thus $C_\lambda$ is
  measurable with respect to the restriction of $\eta$ outside
  $B_R(0)$ for every $R>0$.

  The restrictions to disjoint Euclidean annuli are independent, so
  their tail $\sigma$-algebra is trivial. Each event that $C_\lambda$
  misses an open set therefore has probability zero or one. A
  countable base of open sets determines a closed set; intersecting
  the corresponding countably many probability-one events shows that
  $C_\lambda$ equals a deterministic closed cone almost surely.
\end{proof}

For $u\in\R^d$ and $r>0$, define
\begin{displaymath}
  E_u(r):=\{x\in\R^d:\langle x,u\rangle\geq r\}
\end{displaymath}
and
\begin{displaymath}
  L_\lambda:=\{u:\lambda(E_u(r))<\infty
  \text{ for all sufficiently large }r\}.
\end{displaymath}

\begin{lemma}\label{l:recession-polar}
  The set $L_\lambda$ is a convex cone and $C_\lambda=(L_\lambda)^o$.
\end{lemma}
\begin{proof}
  The origin belongs to $L_\lambda$, since $E_0(r)=\varnothing$ for
  every $r>0$. Moreover, for $s>0$, we have $E_{su}(r)=E_u(r/s)$, and
  $E_{au+(1-a)v}(r)\subset E_u(r)\cup E_v(r)$ 
  for $0\leq a\leq1$.
  It follows that $L_\lambda$ is a convex cone.

  Let $D:=D(\overline{P})$ be the barrier cone of $\overline{P}$. For
  each fixed deterministic $u\in\R^d$, we claim that
  \begin{equation}
    \label{eq:barrier-fixed-direction}
    u\in L_\lambda
    \quad\text{if and only if} \quad
    \BP(u\in D)=1.
  \end{equation}
  Suppose first that $u\in L_\lambda$. Then $\lambda(E_u(r))<\infty$
  for some sufficiently large $r$. Hence $\eta(E_u(r))<\infty$ almost
  surely. All points of $\eta$ outside $E_u(r)$ have scalar product
  with $u$ strictly smaller than $r$, while there are only finitely
  many points in $E_u(r)$. Therefore,
  $h(\overline{P},u)=\sup_{x\in\eta}\langle x,u\rangle<\infty$ almost
  surely.

  Conversely, suppose that $u\notin L_\lambda$. Since
  $r\mapsto E_u(r)$ is decreasing, finiteness of $\lambda(E_u(r))$ for
  one value of $r$ would imply finiteness for all larger values. Thus
  $\lambda(E_u(r))=\infty$ for every sufficiently large $r$, and in
  particular for every sufficiently large positive integer $n$. Hence,
  $\eta(E_u(n))=\infty$ almost surely for each such $n$. Taking the
  countable intersection of these probability-one events yields
  $h(\overline{P},u)=\infty$ almost surely. This proves
  \eqref{eq:barrier-fixed-direction}.

  By Lemma~\ref{l:recession-deterministic},
  $\rec(\overline{P})=C_\lambda$ is deterministic almost surely. By
  the duality relation between the barrier and recession cones
  \cite[Corollary~14.2.1]{MR274683},
  $\cl D=(\rec\overline{P})^o=C_\lambda^o$ almost surely.  Since $D$
  is convex, Rockafellar \cite[Theorem~6.3]{MR274683} gives the
  following equality for the relative interiors:
  \begin{displaymath}
    \operatorname{ri}D
    =\operatorname{ri}(\cl D)
    =\operatorname{ri}(C_\lambda^o)
  \end{displaymath}
  almost surely.
  
  We now compare the two deterministic cones $L_\lambda$ and
  $C_\lambda^o$. Let $u\in L_\lambda$ be fixed. By
  \eqref{eq:barrier-fixed-direction}, $u\in D$ almost surely. Since
  $D\subset\cl D=C_\lambda^o$ almost surely and $C_\lambda^o$ is
  deterministic, it follows that $u\in C_\lambda^o$.  Thus
  $L_\lambda\subset C_\lambda^o$.

  For the converse inclusion up to closure, fix
  $u\in\operatorname{ri}(C_\lambda^o)$.  By the equality of relative
  interiors above, $u\in\operatorname{ri}D\subset D$ almost
  surely. Hence \eqref{eq:barrier-fixed-direction} implies
  $u\in L_\lambda$. We have therefore obtained the deterministic
  inclusions
  \begin{displaymath}
    \operatorname{ri}(C_\lambda^o)
    \subset L_\lambda
    \subset C_\lambda^o.
  \end{displaymath}
  By Rockafellar \cite[Theorem~6.3]{MR274683}, the relative interior
  of a nonempty convex set is dense in that set.  Since $C_\lambda^o$
  is closed, $\cl\operatorname{ri}(C_\lambda^o)=C_\lambda^o$.
  Consequently, $\cl L_\lambda=C_\lambda^o$.  Taking polars, and using
  the fact that polarity is unchanged by taking closure, gives
  \begin{displaymath}
    (L_\lambda)^o
    =(\cl L_\lambda)^o
    =(C_\lambda^o)^o.
  \end{displaymath}
  Finally, $C_\lambda$ is a closed convex cone, so the bipolar theorem
  yields $(C_\lambda^o)^o=C_\lambda$.  Therefore,
  $C_\lambda=(L_\lambda)^o$, as claimed.
\end{proof}

\section*{Acknowledgement}

The authors are grateful to Ferenc Fodor and Viktor Vigh for
suggesting an idea for the proof of the stopping set property in the
ball hulls setting. Furthermore, they thank Christoph Th\"ale for
suggesting the tessellation example, Matthias Reitzner for stimulating
questions that led us to add Section~\ref{sec:ident-based-thinn}. 

In preparing the second version of this paper the authors benefited
from ChatGPT. The AI proposed the counterexample in
Section~\ref{sec:counter} and has checked proofs. ChatGPT also
proposed the proof of Theorem~\ref{t:spatMarkov2} in the special case,
when $Z$ is a stopping set.

\end{document}